\documentclass[11pt]{amsart}  
\usepackage{standard}
\usepackage{biblatex}
\usepackage{etoolbox}
\usepackage{float}
\usepackage{exparrow}
\usepackage{indent-toc-amsart}
\title{Repetitions of Pak-Stanley Labels in $G$-Shi Arrangements}

\author{Cara Bennett}
\address{School of Mathematics, Georgia Tech}
\email{\textcolor{blue}{\href{mailto:cara.bennett@gatech.edu}{cara.bennett@gatech.edu}}}

\author{Lucy Martinez}
\address{Department of Mathematics, Rutgers University}
\email{\textcolor{blue}{\href{mailto:lucy.martinez@rutgers.edu}{lucy.martinez@rutgers.edu}}}

\author{Ava Mock}
\address{Department of Mathematics, Wellesley College}
\email{\textcolor{blue}{\href{mailto:am6@wellesley.edu}{am6@wellesley.edu}}}

\author{Gordon Rojas Kirby}
\address{ Department of Mathematics and Statistics, San Diego State University}
\email{\textcolor{blue}{\href{mailto: gkirby@sdsu.edu}{gkirby@sdsu.edu}}}

\author{Robin Truax}
\address{Department of Mathematics, Stanford University}
\email{\textcolor{blue}{\href{mailto:truax@stanford.edu}{truax@stanford.edu}}}

\date{\today}

\begin{document}
\begin{abstract}
Given a simple graph $G$, one can define a hyperplane arrangement called the $G$-Shi arrangement. The Pak-Stanley algorithm labels the regions of this arrangement with $G_\bullet$-parking functions. When $G$ is a complete graph we recover the full Shi arrangement, and the Pak-Stanley labels give a bijection with ordinary parking functions. However, for proper subgraphs $G \subset K_n$, while the Pak-Stanley labels still include every $G_{\bullet}$-parking function, some appear more than once. These repetitions of Pak-Stanley labels are a topic of interest in the study of $G$-Shi arrangements and $G_{\bullet}$-parking functions. Furthermore, $G_{\bullet}$-parking functions are connected to many other combinatorial objects (for example, superstable configurations in chip-firing). In studying these repetitions, we can draw on existing results about these objects such as Dhar's Burning Algorithm. Conversely, our results have implications for the study of these objects as well. \\

The key insight of our work is the introduction of a combinatorial model called the Three Rows Game. Analyzing the histories of this game and the ways in which they can induce the same outcomes allows us to characterize the multiplicities of the Pak-Stanley labels. Using this model, we develop a classification theorem for the multiplicities of the Pak-Stanley labels of the regions in the $P_n$-Shi arrangement, where $P_n$ is the path graph on $n$ vertices. Then, we generalize the Three Rows Game into the $T$-Three Rows Game. This allows us to study the multiplicities of the Pak-Stanley labels of the regions in $T$-Shi arrangements, where $T$ is any tree. Finally, we discuss the possibilities and difficulties in applying our method to arbitrary graphs. In particular, we analyze multiplicities in the case when $G$ is a cycle graph, and prove a uniqueness result for maximal $G_{\bullet}$-parking functions for all graphs using the Three Rows Game.
\end{abstract}

\vspace{-1cm}
\maketitle
\vspace{-0.5cm}
\section{Introduction}
A \textit{hyperplane} in $\R^n$ is an affine subspace of dimension $n-1$. A \textit{hyperplane arrangement} is a collection of finitely many hyperplanes. 
Given a hyperplane arrangement $\mathcal{A}$, its complement $\R^n \setminus \bigcup_{H\in \mathcal{A}}H$ splits into connected components called \textit{regions}. 
The focus of this paper is on the combinatorial properties of the regions of the \textit{$G$-Shi arrangement} and their labels by $G_\bullet$-parking functions, as defined by Duval, Klivans, and Martin \cite{GShiDefinition}. The $G$-Shi arrangement $\mathscr{S}(G)$ is defined for any graph $G = (V,E)$ by
$$\mathscr{S}(G)=\{ x_i - x_j = 0,1 \mid \{i,j\}\in E \text{ with } i<j\}. $$
When $G = K_n$, the $G$-Shi arrangement is the \textit{Shi arrangement}, which has $(n+1)^{n-1}$ regions. See \cite{ShiSurvey_Fishel} for a survey of the Shi arrangement. 

Since the regions of the Shi arrangement are equinumerous with parking functions of length $n$, a natural problem is to find a bijection between regions of the Shi arrangement and parking functions of length $n$. Pak and Stanley (1996) and
Athanasiadis and Linusson (1999) gave two such bijections \cite{Stanley},\cite{ATHANASIADIS}. The Pak-Stanley algorithm, defined later, labels the regions of the $G$-Shi arrangement with $G_\bullet$-parking functions in such a way that every $G_\bullet$-parking function appears. When $G$ is complete, each $G_\bullet$-parking function appears exactly once as a Pak-Stanley label on a region of the $G$-Shi arrangement. If $G$ is not complete, some $G_\bullet$-parking functions appear more than once as Pak-Stanley labels in the $G$-Shi arrangement. Consider the examples of $G = K_3$ and $G = P_3$ in Figures \ref{fig:labeled-G-Shi-K_3} and \ref{fig:labeled-G-Shi-P_3} respectively. 

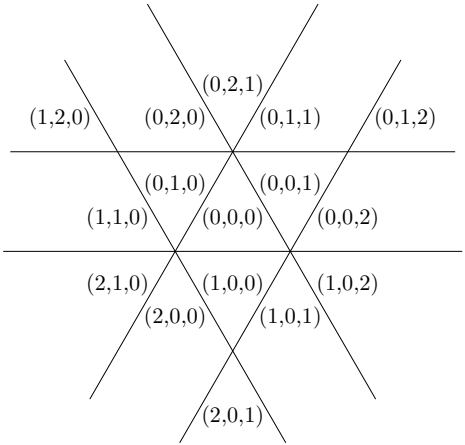
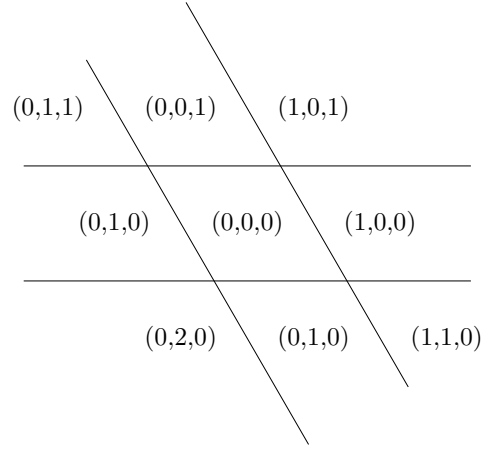
\begin{figure}[H]
    \centering
    \begin{subfigure}[b]{0.40\textwidth}
        \centering
        \resizebox{\linewidth}{!}{






\begin{tikzpicture}[]
\node[] (0) at (-1.561, -2.704) {};
\node[] (1) at (2.561, 4.436) {};
\node[] (2) at (0, -3.464) {};
\node[] (3) at (4, 3.464) {};
\node[] (4) at (-3, 1.732) {};
\node[] (5) at (5, 1.732) {};
\node[] (6) at (-3.123, 0) {};
\node[] (7) at (5.123, 0) {};
\node[] (8) at (-2, 3.464) {};
\node[] (9) at (2, -3.464) {};
\node[] (10) at (-0.561, 4.436) {};
\node[] (11) at (3.561, -2.704) {};

\node[] () at (1, 0.577) {(0,0,0)};
\node[] () at (0, 1.155) {(0,1,0)};
\node[] () at (2, 1.155) {(0,0,1)};
\node[] () at (1, -0.577) {(1,0,0)};

\node[] () at (-1, 0.577) {(1,1,0)}; 
\node[] () at (0, -1.155) {(2,0,0)}; 
\node[] () at (2, -1.155) {(1,0,1)}; 
\node[] () at (3, 0.577) {(0,0,2)}; 
\node[] () at (2, 2.309) {(0,1,1)}; 
\node[] () at (0, 2.309) {(0,2,0)}; 

\node[] () at (-1, -0.577) {(2,1,0)};
\node[] () at (1, -2.877) {(2,0,1)};
\node[] () at (3, -0.577) {(1,0,2)};
\node[] () at (4, 2.309) {(0,1,2)};
\node[] () at (1, 2.877) {(0,2,1)};
\node[] () at (-2, 2.309) {(1,2,0)};

\path [-](0) edge (1);
\path [-](2) edge (3);
\path [-](4) edge (5);
\path [-](6) edge (7);
\path [-](8) edge (9);
\path [-](10) edge (11);

\end{tikzpicture}}
        \caption{Labelling the $K_3$-Shi arrangement.}
        \label{fig:labeled-G-Shi-K_3}
    \end{subfigure}
    \hfill
    \begin{subfigure}[b]{0.40\textwidth}
        \centering
        \resizebox{\linewidth}{!}{\begin{tikzpicture}[]
\node[] (0) at (-3, 0) {};
\node[] (1) at (4, 0) {};
\node[] (2) at (-3, 1.732) {};
\node[] (3) at (4, 1.732) {};
\node[] (4) at (-2, 3.464) {};
\node[] (5) at (1.5, -2.598) {};
\node[] (6) at (-0.5, 4.33) {};
\node[] (7) at (3, -1.732) {};

\node[] () at (-2.5, 2.598) {(0,1,1)};
\node[] () at (-0.5, 2.598) {(0,0,1)};
\node[] () at (1.5, 2.598) {(1,0,1)};
\node[] () at (-1.5, 0.866) {(0,1,0)};
\node[] () at (0.5, 0.866) {(0,0,0)};
\node[] () at (2.5, 0.866) {(1,0,0)};
\node[] () at (-0.5, -0.866) {(0,2,0)};
\node[] () at (1.5, -0.866) {(0,1,0)};
\node[] () at (3.5, -0.866) {(1,1,0)};

\path [-](0) edge (1);
\path [-](2) edge (3);
\path [-](4) edge (5);
\path [-](6) edge (7);

\end{tikzpicture}}
        \caption{Labelling the $P_3$-Shi arrangement.}
        \label{fig:labeled-G-Shi-P_3}
    \end{subfigure}
    \caption{Example Pak-Stanley Labels for $G$-Shi Arrangements}
\end{figure}

Figures \ref{fig:labeled-G-Shi-K_3} and \ref{fig:labeled-G-Shi-P_3} depict $G$-Shi arrangements. Notice that on the left, each label appears once (has multiplicity 1), but on the right, the label $(0,1,0)$ appears twice (has multiplicity 2). In this paper, we work towards a characterization of the ``multiplicities" of the Pak-Stanley labels in $G$-Shi arrangements by introducing a combinatorial model called the Three Rows Game defined in Section \ref{sec:three-rows-game}

In Section \ref{sec:three}, we introduce the Shi adjacency digraph to describe the adjacencies of regions in the $G$-Shi arrangement and their relationship to the Pak-Stanley labels of these regions. Section \ref{sec:four} introduces the Three Rows game, a combinatorial model for analyzing multiplicities of Pak-Stanley labels in $P_n$-Shi arrangements for path graphs $P_n$. In particular, we prove the following result to characterize multiplicities of the Pak-Stanley labels in $G$-Shi arrangements for path graphs.
\begin{theorem*}[Path Multiplicity Theorem]
Suppose ${\bf p} = (p_1,\dots,p_n)$ is a
Pak-Stanley label of a region of $\mathscr{S}(P_n)$. A \textit{run} ${\bf r}$ of length $k$ in ${\bf p}$ is a section of ${\bf p}$ of the form $(0,1,\dots,1,0)$ with $k$ $1$s. If the length of a run ${\bf r}$ is denoted $\ell({\bf r})$, then the multiplicity of the label ${\bf p}$ in in the $P_n$-Shi arrangement is
\begin{equation*}
    \mu({\bf p}) = \prod_{\text{runs ${\bf r}$ in ${\bf p}$}} (\ell({\bf r})+1).
\end{equation*}
\end{theorem*} 
Besides analyzing path graphs, we also study multiplicities of Pak-Stanley labels in $G$-Shi arrangements for all trees in Section \ref{sec:five}, for example characterizing multiplicities for star graphs with a similar theorem. Then in Section \ref{sec:six}, we discuss how the Three Rows Game can be played on all graphs, not just on trees, and the added complexity of such an extension. In particular, we completely characterize multiplicities of the Pak-Stanley labels in $G$-Shi arrangements for cycle graphs with an analogous Cycle Multiplicity Theorem. Finally, we prove the following fact about multiplicities of the maximal labels of regions of $G$-Shi arrangements. 

\begin{corollary*}
Let $G$ be a graph and ${\bf p}$ a maximal $G_\bullet$-parking function. Then $\bf p$ has multiplicity $1$.
\end{corollary*}
\section{Background}
We begin with a discussion of some prerequisite topics: the $G$-Shi arrangement, the Pak-Stanley algorithm, chip-firing, and superstable configurations. An important note is that for the remainder of the paper, all of our graphs are assumed to be connected, finite, and simple.
\subsection{The $G$-Shi Arrangement and the Pak-Stanley Algorithm}

First, we formally define the $G$-Shi arrangement.

\begin{definition}[Shi Arrangement]\label{def:shi-arrangement}
The \textit{Shi arrangement} $\mathscr{S}_n$ is the hyperplane arrangement in $\R^n$ with hyperplanes $x_i-x_j=0$ and $x_i-x_j=1$ for each $i,j\in\{0,1,2,\dots,n-1\}$ with $i<j$.
\end{definition}
\begin{definition}[$G$-Shi Arrangement]
Given a graph $G = (V,E)$ with $V = \{0,\dots,n-1\}$, the \textit{$G$-Shi arrangement} $\mathscr{S}(G)$ is the hyperplane arrangement in $\R^n$ with hyperplanes $x_i-x_j=0$ and $x_i-x_j=1$ for each $\{i,j\}\in E$ with $i<j$ \cite{Hopkins-Perkinson}.  In this case, $G$ is called the \textit{defining graph} of the arrangement.
\end{definition}

It is important to note that the $G$-Shi arrangement depends on the labelling $0,\dots,n-1$ of the vertices of $G$. Indeed, consider the isomorphic graphs $K_4 \setminus \{0,1\}$ and $K_4 \setminus \{0,2\}$. The former has a $G$-Shi arrangement with 84 regions, whereas the latter has a $G$-Shi arrangement with 85 regions. Over the course of this paper, we will demonstrate that in certain special cases the multiplicities do not depend on the labelling of the vertices.

Figure \ref{fig:labeled-G-Shi-K_3} illustrates a projection of the Shi arrangement onto the hyperplane $x_0 + x_1 + x_2 = 0$, allowing us to visualize it in $2$ dimensions without losing any regions or adjacency relations. This projection trick is also used in Figure \ref{fig:labeled-G-Shi-P_3}.



Next, we discuss $G$-parking functions and the Pak-Stanley algorithm (which surjectively assigns $G_\bullet$-parking functions to the regions of the $G$-Shi arrangement).

\begin{definition}[Outdegree]\label{def:outdegree}
Given a graph $G = (V,E)$, a subset $S$ of $V$, and a vertex $v$, the \textit{outdegree $\outdeg_S(v)$ of $v$ with respect to $S$} is the number of edges from $v$ to vertices outside $S$. In particular, $\outdeg_{\varnothing}(v)$ is the outdegree of $v$ (which equals $\deg(v)$ when $G$ is undirected).
\end{definition}

\begin{definition}[$G$-Parking Function]\label{def:G-parking-function}
Let $G$ be an undirected graph on vertices $V=\{0,1,\hdots,n-1,q\}$. In this case, $q$ is called the \textit{sink}. A \textit{$G$-parking function} is an $n$-tuple $(a_0,a_1,\hdots,a_{n-1})$ such that for any non-empty subset $S \subseteq \{0,\dots,n-1\}$, there exists $v\in S$ such that $a_v < \outdeg_S(v)$. 
\end{definition}

\begin{definition}[$G_{\bullet}$]\label{def:G_bullet}
Let $G=(V,E)$ be a graph. Then, $G_{\bullet}$ is the graph obtained from $G$ by adding a \textit{sink} vertex $q$ and an edge between $q$ and each $v\in V$.
\end{definition}



Next, we will discuss the Pak-Stanley algorithm \cite{Stanley}, whose behavior is the central topic of our paper.

\begin{definition}[Pak-Stanley Algorithm]\label{def:pak-stanley-algorithm}
The \textit{Pak-Stanley algorithm} maps $G_{\bullet}$-parking functions to the regions of the $G$-Shi arrangement \cite{Stanley}. It assigns each region $R$ an $n$-tuple $\lambda(R)$ of nonnegative integers, called its \textit{Pak-Stanley label} as follows: 
\begin{enumerate}
    \item The region in which $x_0>x_1>\dots>x_{n-1}$ and $x_0-x_{n-1}<1$ is called the \textit{base region} and denoted $R_0$. Define $\lambda(R_0)=(0,0,\dots,0)$.
    \item Suppose $\lambda(R)$ has been defined, and that $R'$ is a region such that 
    \begin{enumerate}
        \item $R$ and $R'$ share a boundary facet, which is part of a hyperplane $H \in \mathscr{S}(G)$.
        \item $R$ lies in the same half-space of $H$ as $R_0$.
        \end{enumerate}
    In this case, we define
    $\lambda(R') = 
    \begin{cases}
     \lambda(R) + e_i & \text{if } H \text{ is given by } x_i-x_j=0 \text{ with } i<j,\\
     \lambda(R) + e_j &  \text{if } H\text{ is given by }x_i-x_j=1\text{ with }i<j.
    \end{cases}$
\end{enumerate}

Here, $e_i$ denotes the $i$th standard basis vector of $\R^n$.

The set of Pak-Stanley labels for the regions of the $G$-Shi arrangement are called \textit{the Pak-Stanley labels for $G$}; these are the same as the $G_\bullet$-parking functions \cite{Hopkins-Perkinson}.
\end{definition}

The examples of $G = K_3$ and $G = P_3$ are displayed in Figures \ref{fig:labeled-G-Shi-K_3} and \ref{fig:labeled-G-Shi-P_3}. Notice every $(K_3)_\bullet$-parking function appears in Figure \ref{fig:labeled-G-Shi-K_3}, and every $(P_3)_\bullet$-parking function appears in Figure \ref{fig:labeled-G-Shi-P_3}. Furthermore, each $(K_3)_\bullet$-parking function appears exactly once in the $K_3$-Shi arrangement. As it turns out, both of these are general patterns that are proved in \cite{Hopkins-Perkinson}. Specifically, we have the following:

\begin{theorem}[\cite{Hopkins-Perkinson}, Corollary 2.8]\label{thm:every-G-parking-functions-occurs-as-label}
Every $G_\bullet$-parking function occurs as a label in the Pak-Stanley algorithm on the $G$-Shi arrangement.
\end{theorem}
It is well known that there are $(n+1)^{n-1}$ regions in the Shi arrangement $\mathscr{S}_n$. Similarly, there are $(n+1)^{n-1}$ parking functions (and therefore $(n+1)^{n-1}$ $(K_n)_\bullet$-parking functions). Thus, every $(K_n)_\bullet$-parking function occurs as a label precisely once in the Pak-Stanley algorithm on the $K_n$-Shi arrangement.

\subsection{Chip-Firing, Superstable Configurations, and Dhar's Burning Algorithm}
We begin this section by recounting some of the definitions in Chapter 2 of \textit{The Mathematics of Chip-firing} by Klivans \cite{Klivans}. For more detail, we refer the reader to this book.

In this section, $G = (V,E)$ is a graph on $n+1$ vertices with sink vertex $q$.

\begin{definition}[Chip Configuration]\label{def:chip-configuration}
A \textit{chip configuration} for $G$ is a non-negative integer vector
\begin{equation*}
    {\bf c} = (c_0,c_1,\dots,c_{n-1}) \in \Z^n_{\geq 0}
\end{equation*}
where each coordinate of the vector corresponds to the number of ``chips" at a particular vertex. That is, the $i$th coordinate $c_i$ represents the number of chips at the vertex $v_i$. Notice that we do not consider the number of chips on the sink vertex $q$.
\end{definition}

\begin{definition}[Firing]\label{def:firing}
Given a chip configuration ${\bf c}$ for $G$, a non-sink vertex $v$ \textit{fires} by sending one chip to each of its neighbors (possibly including the sink). That is, the configuration ${\bf c}$ is replaced by ${\bf c'} = (c_0',\dots,c_{n-1}') \in \Z^n$ where
\begin{equation*}
    c_i' = \begin{cases}
     c_i - \deg i & i = v \\
     c_i + 1 & \{i,v\} \in E \\
     c_i & \text{otherwise.}
    \end{cases}
\end{equation*}
This is a \textit{legal fire} if ${\bf c'}$ is a chip configuration.
\end{definition}

\begin{definition}[Stable]\label{def:stable}
A chip configuration ${\bf c}$ is \textit{stable} if there are no legal fires.
\end{definition}

Now let us transition to discussing \textit{superstable configurations}. 

\begin{definition}[Graph Laplacian]\label{def:the-graph-laplacian}
Let $G=(V,E)$ be a graph on $n$ vertices $v_0,\dots,v_{n-1}$. The graph \textit{Laplacian} $\Delta(G)$ is the $n\times n$ matrix given by
\begin{equation*}
    \Delta_{ij} = \begin{cases}
        \deg(v_i) & i = j\\
        -1 & i\neq j \text{ and } \{v_i,v_j\} \in E \\
        0 & \text{otherwise.}
    \end{cases}
\end{equation*}
\end{definition}

\begin{definition}[Reduced Laplacian]\label{def:reduced-laplacian}
Let $\Delta$ be the Laplacian of a graph $G$ with sink $q$. Then, the \textit{reduced graph Laplacian of $G$ with respect to $q$}, denoted $\Delta_q(G)$ or $\Delta_q$, is the matrix obtained from $\Delta$ by deleting the row and column corresponding to $q$.
\end{definition}


\begin{definition}[Cluster-Fire]\label{def:cluster-fire}
Let $G$ be a graph on $n+1$ vertices with sink $q$. Consider a chip configuration ${\bf c} = (c_0,\dots,c_{n-1})$. Let $S$ be a subset of the vertices of $G$. A \textit{cluster-fire at $S$} replaces ${\bf c}$ with a new configuration ${\bf c}'$ given by sending a chip to each neighbor of $v$ for each $v \in S$. Formally, 
\begin{equation*}
    {\bf c}' = {\bf c} - \Delta_q(G)\chi_S
\end{equation*}
where $\chi_S \in \R^n$ is the characteristic vector of $S$; that is, $\chi_S$ has $i$th coordinate equal to $1$ if $i \in S$ and $0$ if $i \not\in S$. Such a cluster-fire is called \textit{legal} if ${\bf c'}$ is still a chip configuration.
\end{definition}

\begin{definition}[Superstable Configuration]\label{def:superstable-configuration}
Let a graph $G$ be given with a chip configuration ${\bf c}$. The configuration ${\bf c}$ is called \textit{superstable} if there are no legal cluster-fires from ${\bf c}$.
\end{definition}

\begin{example}
The left-hand configuration is not superstable, as one can fire at the three vertices whose chip labels are bolded. However, the right-hand side is superstable.

\begin{figure}[H]
    \centering
    \includegraphics[width=0.5\textwidth]{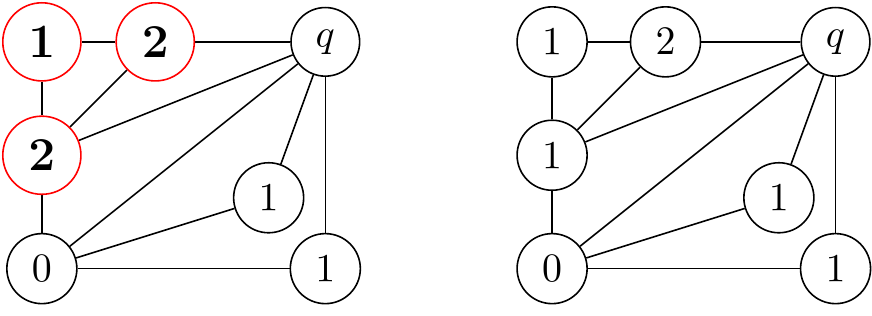}
    \caption{Non-Superstable and Superstable Configurations.}
    \label{fig:Non-Superstable-and-Superstable-Configurations}
\end{figure}
\end{example}
Now, the main reason why superstable configurations are central to our paper is the following result: 

\begin{theorem}[\cite{Klivans}, Theorem 3.6.3]\label{thm:superstable-configurations-are-G-parking-functions}
If $G$ is a graph with sink vertex $q$, then the $G$-parking functions of $G$ are precisely the set of superstable configurations of $G$.
\end{theorem}
\begin{proof}
By definition; see Definitions \ref{def:G-parking-function} and \ref{def:superstable-configuration}.
\end{proof}

We conclude with a discussion of some important results about superstable configurations. First, we will discuss Dhar's Burning Algorithm. 

\begin{definition}[Dhar's Burning Algorithm]\label{def:dhars-burning-algorithm}
Let $G$ be a graph with sink $q$, and ${\bf c}$ be a chip configuration on $G$. Envision the chips as firefighters protecting the vertices that they are on. Then, light the sink $q$ on fire, and repeat the following process: 
\begin{enumerate}
    \item If a vertex $v$ is on fire, the fire spreads along all the edges of $v$ towards $v$'s neighbors.
    \item Each firefighter can stop the fire along exactly one edge.
    \item A vertex $v'$ does not light on fire as long as there is a firefighter stopping the fire at each burning incident edge of $v'$. However, if there are more incident burning edges at $v'$ than firefighters, $v'$ catches on fire.
\end{enumerate}
\end{definition}


\begin{theorem}[\cite{Klivans}, Theorem 2.6.24]\label{def:dhars-burning-algorithm-detects-superstability}
Let $G$ be a finite graph with sink $q$. Start a fire at $q$. A configuration is superstable if and only if every vertex is eventually on fire.
\end{theorem}
\begin{proof}
Begin with a chip configuration ${\bf c}$ and let $S$ denote the set of vertices which are not on fire when the algorithm concludes. If $S \neq \varnothing$, then there is a legal cluster-fire at $S$. To see why, notice that each vertex $v \in S$ must start with $\outdeg_S(v)$ firefighters to prevent itself from lighting on fire, and it loses precisely $\outdeg_S(v)$ chips in a cluster-fire at $S$. On the other hand, suppose $S = \varnothing$. Then, there is no legal cluster-fire at any $U \subseteq V \setminus \{q\}$. For, if $u$ is the first vertex of $U$ to catch on fire, ${\bf c}_u < \outdeg_U(u)$, so a cluster-fire at $U$ would leave $u$ with negative chips.
\end{proof}

Next, we discuss a powerful concept called critical-superstable duality and its implications.


\begin{definition}[Critical Configurations]\label{def:critical-configuration}
A chip configuration ${\bf c}$ is called \textit{critical} if it is stable and arises from the chip-firing process from a configuration ${\bf c}'$ in which every vertex is ready to fire.
\end{definition}

\begin{theorem}[\cite{Klivans}, Theorem 2.6.19]\label{prop:critical-superstable-duality}
Let ${\bf c}_{\text{max}}$ denote the configuration given by $({\bf c}_{\text{max}})_i = \deg(i) - 1$. Then, for any graph $G$, ${\bf c}$ is critical if and only if ${\bf c}_{\text{max}} - {\bf c}$ is superstable.
\end{theorem}

\begin{proposition}[Chips in Maximal Superstable Configurations]\label{prop:maximal-superstable-conf}
Suppose that $G$ is a graph on $n$ vertices with $m$ edges. Then every maximal superstable configuration on $G_\bullet$ has $m$ chips. Conversely, every superstable configuration on $G_\bullet$ with $m$ chips is maximal.
\end{proposition}
\begin{proof}
Since $G_\bullet$ has $m+n$ edges and $\deg(q) = n$, Proposition 2.6.22 of Klivans yields that any minimal critical configuration on $G_\bullet$ has $m$ chips \cite{Klivans}. On the other hand, the sum of the degrees of all non-sink vertices of $G_\bullet$ is $2m+n$, and to compute ${\bf c}_{\text{max}}$ we subtract $1$ for each of the $n$ non-sink vertices, then $c_{\text{max}}$ on $G_\bullet$ has $2m$ chips. Therefore, by Theorem \ref{prop:critical-superstable-duality}, we have that every maximal superstable configuration has $2m - m = m$ chips. The converse follows immediately.
\end{proof}

\section{The Shi Adjacency Digraph}\label{sec:three}

In this section, we introduce the Shi adjacency digraph and discuss how we use it to perform the Pak-Stanley algorithm.

\subsection{Shi Adjacency Digraphs}
When following the Pak-Stanley algorithm, we assign labels to new regions based on the labels of adjacent regions. 
The order in which we assign the labels depends on the placement of the regions with respect to the base region $R_0$. 
To encode this information about region adjacency and the order of assigning labels, we created the Shi adjacency digraph. 

\begin{definition}[Shi Adjacency Digraph]
Given a graph $G$, called the \textit{defining graph}, we construct the \textit{Shi adjacency digraph} $\Gamma\mathscr{S}(G)$ from the $G$-Shi arrangement as follows: 
\begin{enumerate}
    \item Each region $R_i$ in the $G$-Shi arrangement is mapped to the vertex $v_i$ in $\Gamma\mathscr{S}(G)$.
    \item Two vertices $v_i$ and $v_j$ are connected by a directed edge if and only if they correspond to regions $R_i$ and $R_j$ which share a face given by a hyperplane $H$. 
    \begin{enumerate}
        \item If $R_0$ and $R_i$ lie on the same side of $H$, the directed edge is $(v_i,v_j)$. 
        \item If $R_0$ and $R_j$ lie on the same side of $H$, the directed edge is $(v_j,v_i)$.
    \end{enumerate}
\end{enumerate}
\end{definition}

Another way to think about the Shi adjacency digraph is to realize that the $G$-Shi arrangement endows $\R^n$ with a natural cell structure, and the Shi adjacency digraph (or rather, the undirected graph underlying it) is the $1$-skeleton of the dual of that cell structure. 

We can also label each of the directed edges with a value corresponding to the Pak-Stanley algorithm so that we can recover the Pak-Stanley labels of vertices directly from the graph. If a directed edge corresponds to hyperplane $x_i-x_j=c$ separating adjacent regions then we label it $i$ if $c=0$ or $j$ if $c=1$. Then, we label the source vertex $(0,0,\dots, 0)$, and the label of $v_j$ is equal to the label of $v_i$ incremented by 1 at the index assigned to the edge $(v_i,v_j)$. 

To generate these graphs, we created the SageMath class \lstinline{G_Shi}. The class takes in a defining graph $G$, constructs the Shi adjacency digraph, and uses the above algorithm to compute the Pak-Stanley labels. The code and documentation for this program is in Appendix \ref{appendix:g_shi_code}.

\begin{figure}[H]
    \centering
    \begin{subfigure}[b]{0.49\textwidth}
        \centering
        \resizebox{.95\linewidth}{!}{\begin{tikzpicture}[scale=2, font=\scriptsize]
\node[] (0) at (-1.645, -0.848) {};
\node[] (1) at (0.49, 2.848) {};
\node[] (2) at (-0.817, -1.415) {};
\node[] (3) at (1.394, 2.415) {};
\node[] (4) at (0.817, -1.415) {};
\node[] (5) at (-1.394, 2.415) {};
\node[] (6) at (1.645, -0.848) {};
\node[] (7) at (-0.49, 2.848) {};
\node[] (8) at (-2.211, 1) {};
\node[] (9) at (2.211, 1) {};
\node[] (10) at (-2.134, 0) {};
\node[] (11) at (2.134, 0) {};

\path [-](0) edge (1);
\path [-](2) edge (3);
\path [-](4) edge (5);
\path [-](6) edge (7);
\path [-](8) edge (9);
\path [-](10) edge (11);

\end{tikzpicture}}
        \label{step1er}
    \end{subfigure}
    \begin{subfigure}[b]{0.49\textwidth}
    \centering
    \begin{tikzpicture}[scale=2]
\node[shape=circle,draw=black, fill=black] (1) at (0, 0.667) {};
\node[shape=circle,draw=black, fill=black] (2) at (-0.577, 0.333) {};
\node[shape=circle,draw=black, fill=black] (3) at (0.577, 0.333) {};
\node[shape=circle,draw=black, fill=black] (4) at (0, 1.333) {};
\node[shape=circle,draw=black, fill=black] (5) at (1.155, 0.667) {};
\node[shape=circle,draw=black, fill=black] (6) at (0.577, 1.667) {} ;
\node[shape=circle,draw=black, fill=black] (7) at (-0.577, 1.67) {} ;
\node[shape=circle,draw=black, fill=black] (8) at (-1.155, 0.667) {} ;
\node[shape=circle,draw=black, fill=black] (9) at (-0.577, -0.333) {} ;
\node[shape=circle,draw=black, fill=black] (10) at (0.577, -0.333) {} ;
\node[shape=circle,draw=black, fill=black] (11) at (1.155, 1.333) {} ;
\node[shape=circle,draw=black, fill=black] (12) at (0, 2.667) {} ;
\node[shape=circle,draw=black, fill=black] (13) at (-1.155, 1.333) {} ;
\node[shape=circle,draw=black, fill=black] (14) at (-1.732, -0.333) {} ;
\node[shape=circle,draw=black, fill=black] (15) at (0, -0.667) {} ;
\node[shape=circle,draw=black, fill=black] (16) at (1.732, -0.333) {} ;

\path [->, thick](1) edge node [fill=white] {0} (2);
\path [->, thick](1) edge node [fill=white] {2} (3);
\path [->, thick](1) edge node [fill=white] {1} (4);
\path [->, thick](4) edge node [fill=white] {1}  (6);
\path [->, thick](4) edge node [fill=white] {0}  (7);
\path [->, thick](2) edge node [fill=white] {0} (8) ;
\path [->, thick](2) edge node [fill=white] {2} (9);
\path [->, thick](3) edge node [fill=white] {1} (5);
\path [->, thick](3) edge node [fill=white] {2} (10);
\path [->, thick](5) edge node [fill=white] {1} (11);
\path [->, thick](6) edge node [fill=white] {2} (11);
\path [->, thick](6) edge node [fill=white] {0} (12);
\path [->, thick](7) edge node [fill=white] {1} (12);
\path [->, thick](7) edge node [fill=white] {0} (13);
\path [->, thick](8) edge node [fill=white] {1} (13);
\path [->, thick](8) edge node [fill=white] {2} (14);
\path [->, thick](9) edge node [fill=white] {0} (14);
\path [->, thick](9) edge node [fill=white] {2} (15);
\path [->, thick](10) edge node [fill=white] {0} (15);
\path [->, thick](10) edge node [fill=white] {1} (16);
\path [->, thick](5) edge node [fill=white] {2} (16);

\end{tikzpicture}
    \label{fig:labeled-adjacency-graph}
    \end{subfigure}
    \caption{The $G$-Shi arrangement (left) and its Shi adjacency digraph (right) of the defining graph $K_3$ with edges labelled by the index to be incremented.}
\end{figure}
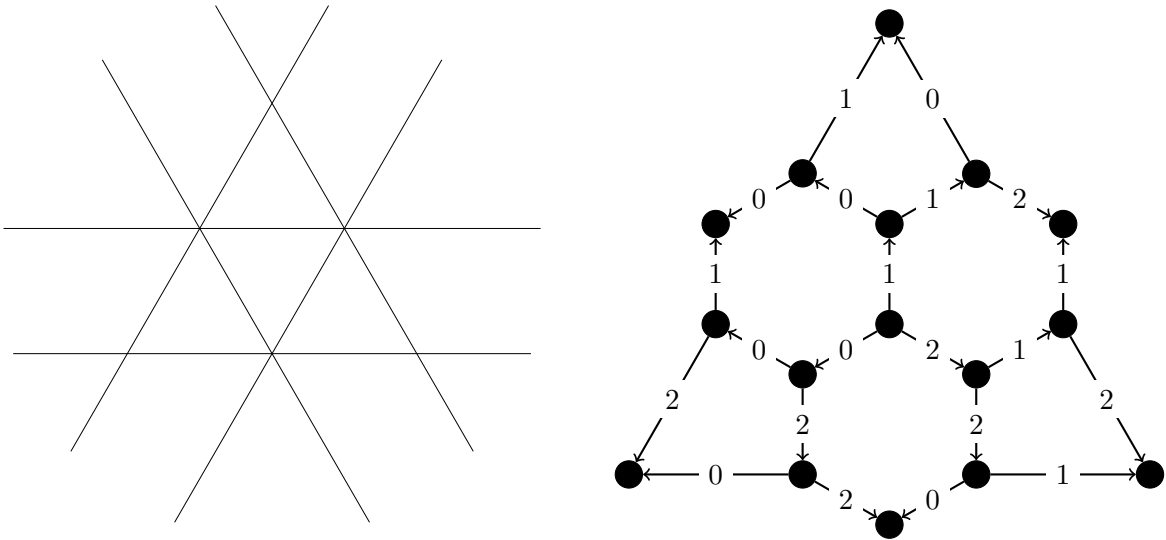



In this paper, Shi adjacency digraphs are mostly used as a computational tool, but they do have a few interesting properties as graphs. $\Gamma \mathscr{S} (G)$ is always acyclic. If $n$ is the number of vertices of $G$, then $\Gamma \mathscr{S} (G)$ is $(n-1)$-edge connected. Finally, we have the following result: 

\begin{theorem}
$\Gamma \mathscr{S}(G)$ is planar if and only if $G$ has $n \leq 3$ vertices.
\end{theorem}
\begin{proof}
If $G$ has at most $3$ vertices, the $G$-Shi arrangement can be represented in $\R^2$ by projecting onto the plane $x_0 + x_1 + x_2 = 0$; planarity is then obvious. On the other hand, one may verify by brute force that $\Gamma \mathscr{S}(G)$ is not planar if $n = 4$. Yet if $H$ is a subgraph of $G$, $\Gamma \mathscr{S}(H)$ is a subdivision of $\Gamma \mathscr{S}(G)$ by construction. Therefore, by Kuratowski's Theorem, if $\Gamma \mathscr{S}(H)$ is not planar, then $\Gamma \mathscr{S}(G)$ is not planar. Thus, if $G$ has more than $4$ vertices, $\Gamma \mathscr{S}(G)$ cannot be planar, as desired.
\end{proof}

\subsection{The Shi Adjacency Digraphs of Trees}\label{sec:SAD-of-trees}

In this section, we consider the structure of the $G$-Shi arrangement and the Shi adjacency digraph when $G$ is a tree. 

When constructing the $G$-Shi arrangement, we include the pair of parallel hyperplanes $x_i-x_j=0$ and $x_i-x_j=1$ for each edge $\{i,j\}$, where $i<j$, in the defining graph $G$. 
Each of these pairs of hyperplanes divides $\mathbb{R}^n$ into three distinct regions: a region $0<x_i-x_j<1$ between the hyperplanes, and two regions $x_i-x_j<0$ and $x_i-x_j>1$ on each side of the hyperplanes. Since trees are acyclic, this choice is made independently for each pair of parallel hyperplanes. Thus, as a tree on $n$ vertices has $n-1$ edges, the result is that the Shi adjacency digraphs of trees take on the form of the subdivided $(n-1)$-dimensional cube $[-1,1]^{n-1}$ with vertices $\{-1,0,1\}^{n-1}$, as demonstrated in Figure \ref{fig:adjacency-trees-all}.

Furthermore, the base region lies between each pair of parallel hyperplanes; that is, it lies in the region $0<x_i-x_j<1$ for each $\{i,j\} \in E$. Therefore, whenever a vertex lies between a pair of parallel hyperplanes (that is, it satisfies $0<x_i-x_j<1$), it has two outward-bound edges: one pointing to a vertex satisfying $x_i-x_j<0$ and one pointing to a vertex satisfying $x_i-x_j>1$. This demonstrates that the edges of the cube are always directed towards the corners.

\begin{figure}[H]
    \centering
    \resizebox{.95\textwidth}{!}{\begin{tikzpicture}
		\node [shape=circle,draw=black, fill=black] (0) at (10.5, -0.25) {};
		\node [shape=circle,draw=black, fill=black] (1) at (12.5, 0.25) {};
		\node [shape=circle,draw=black, fill=black] (2) at (8.5, -0.75) {};
		\node [shape=circle,draw=black, fill=black] (3) at (10, -1.5) {};
		\node [shape=circle,draw=black, fill=black] (4) at (12, -1) {};
		\node [shape=circle,draw=black, fill=black] (5) at (14, -0.5) {};
		\node [shape=circle,draw=black, fill=black] (6) at (7, 0) {};
		\node [shape=circle,draw=black, fill=black] (7) at (9, 0.5) {};
		\node [shape=circle,draw=black, fill=black] (8) at (11, 1) {};
		\node [shape=circle,draw=black, fill=black] (9) at (11, 3) {};
		\node [shape=circle,draw=black, fill=black] (10) at (11, 5) {};
		\node [shape=circle,draw=black, fill=black] (11) at (12.5, 2.25) {};
		\node [shape=circle,draw=black, fill=black] (12) at (12.5, 4.25) {};
		\node [shape=circle,draw=black, fill=black] (13) at (14, 1.5) {};
		\node [shape=circle,draw=black, fill=black] (14) at (14, 3.5) {};
		\node [shape=circle,draw=black, fill=black] (15) at (10.5, 1.75) {};
		\node [shape=circle,draw=black, fill=black] (16) at (10.5, 3.75) {};
		\node [shape=circle,draw=black, fill=black] (17) at (9, 2.5) {};
		\node [shape=circle,draw=black, fill=black] (18) at (9, 4.5) {};
		\node [shape=circle,draw=black, fill=black] (19) at (7, 2) {};
		\node [shape=circle,draw=black, fill=black] (20) at (7, 4) {};
		\node [shape=circle,draw=black, fill=black] (21) at (8.5, 1.25) {};
		\node [shape=circle,draw=black, fill=black] (22) at (8.5, 3.25) {};
		\node [shape=circle,draw=black, fill=black] (23) at (10, 0.5) {};
		\node [shape=circle,draw=black, fill=black] (24) at (12, 1) {};
		\node [shape=circle,draw=black, fill=black] (25) at (12, 3) {};
		\node [shape=circle,draw=black, fill=black] (26) at (10, 2.5) {};
		\node [shape=circle,draw=black, fill=black] (27) at (2, 2) {};
		\node [shape=circle,draw=black, fill=black] (28) at (2, 0) {};
		\node [shape=circle,draw=black, fill=black] (29) at (2, 4) {};
		\node [shape=circle,draw=black, fill=black] (30) at (4, 4) {};
		\node [shape=circle,draw=black, fill=black] (31) at (4, 2) {};
		\node [shape=circle,draw=black, fill=black] (32) at (4, 0) {};
		\node [shape=circle,draw=black, fill=black] (33) at (0, 0) {};
		\node [shape=circle,draw=black, fill=black] (34) at (0, 2) {};
		\node [shape=circle,draw=black, fill=black] (35) at (0, 4) {};
		\node [shape=circle,draw=black, fill=black] (36) at (-3, 2) {};
		\node [shape=circle,draw=black, fill=black] (37) at (-3, 0) {};
		\node [shape=circle,draw=black, fill=black] (38) at (-3, 4) {};
		\node [shape=circle,draw=black, fill=black] (39) at (-6, 2) {};
		\path [->, thick] (0) edge (4);
		\path [->, thick] (0) edge (7);
		\path [->, thick] (7) edge (8);
		\path [->, thick] (0) edge (1);
		\path [->, thick] (4) edge (5);
		\path [->, thick] (4) edge (3);
		\path [->, thick] (0) edge (2);
		\path [->, thick] (7) edge (6);
		\path [->, thick] (2) edge (6);
		\path [->, thick] (2) edge (3);
		\path [->, thick] (1) edge (8);
		\path [->, thick] (1) edge (5);
		\path [->, thick] (11) edge (1);
		\path [->, thick] (9) edge (8);
		\path [->, thick] (11) edge (9);
		\path [->, thick] (11) edge (13);
		\path [->, thick] (13) edge (5);
		\path [->, thick] (13) edge (14);
		\path [->, thick] (11) edge (12);
		\path [->, thick] (9) edge (10);
		\path [->, thick] (12) edge (14);
		\path [->, thick] (12) edge (10);
		\path [->, thick] (18) edge (10);
		\path [->, thick] (16) edge (12);
		\path [->, thick] (15) edge (0);
		\path [->, thick] (15) edge (17);
		\path [->, thick] (15) edge (11);
		\path [->, thick] (17) edge (9);
		\path [->, thick] (15) edge (16);
		\path [->, thick] (17) edge (18);
		\path [->, thick] (18) edge (20);
		\path [->, thick] (17) edge (19);
		\path [->, thick] (15) edge (21);
		\path [->, thick] (19) edge (6);
		\path [->, thick] (19) edge (20);
		\path [->, thick] (21) edge (22);
		\path [->, thick] (21) edge (2);
		\path [->, thick] (21) edge (19);
		\path [->, thick] (22) edge (20);
		\path [->, thick] (16) edge (22);
		\path [->, thick] (16) edge (18);
		\path [->, thick] (17) edge (7);
		\path [->, thick] (24) edge (13);
		\path [->, thick] (15) edge (24);
		\path [->, thick] (24) edge (4);
		\path [->, thick] (23) edge (3);
		\path [->, thick] (21) edge (23);
		\path [->, thick] (24) edge (23);
		\path [->, thick] (22) edge (26);
		\path [->, thick] (25) edge (26);
		\path [->, thick] (23) edge (26);
		\path [->, thick] (24) edge (25);
		\path [->, thick] (16) edge (25);
		\path [->, thick] (25) edge (14);
		\path [->, thick] (27) edge (29);
		\path [->, thick] (27) edge (28);
		\path [->, thick] (29) edge (30);
		\path [->, thick] (28) edge (32);
		\path [->, thick] (27) edge (34);
		\path [->, thick] (29) edge (35);
		\path [->, thick] (28) edge (33);
		\path [->, thick] (34) edge (33);
		\path [->, thick] (34) edge (35);
		\path [->, thick] (31) edge (30);
		\path [->, thick] (31) edge (32);
		\path [->, thick] (27) edge (31);
		\path [->, thick] (36) edge (38);
		\path [->, thick] (36) edge (37);
\end{tikzpicture}}
    \caption{The Shi adjacency digraph of trees on 1, 2, 3, and 4 vertices.}
    \label{fig:adjacency-trees-all}
\end{figure}
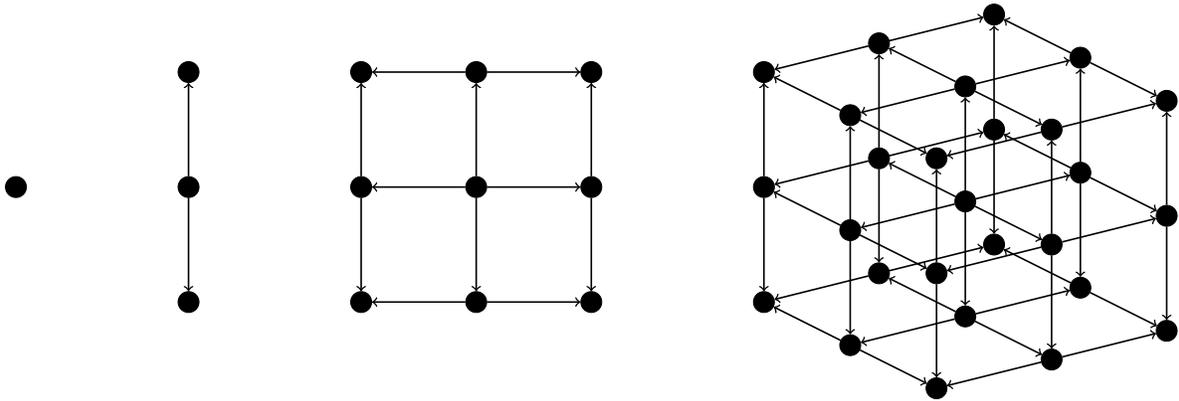


\section{Path Graphs and the Three Rows Game}\label{sec:four}
In this section, we introduce the Three Rows Game, a tool for analyzing multiplicities in the Pak-Stanley labelling. We will use the specific case of the path graph to demonstrate how the Three Rows Game translates the problem of computing Pak-Stanley labels into a combinatorial model. Later, we will generalize the Three Rows Game to be played on trees and general graphs. 

For the remainder of this section, we label the vertices of the path graph $P_n$ as follows unless otherwise stated.
\[
\begin{tikzpicture}
\draw 
    (0,0) to (3.3,0)
    (4.1,0) to (5.9,0);
\fill 
    (0,0) circle (3pt)
    (1.5,0) circle  (3pt)
    (3,0) circle  (3pt)
    (4.4,0) circle (3pt)
    (5.9,0) circle (3pt)
    (3.5,0) circle (1pt)
    (3.7,0) circle (1pt)
    (3.9,0) circle (1pt);
    \node at (0,-0.5) {$0$};
    \node at (1.5,-0.5) {$1$};
    \node at (3,-0.5) {$2$};
    \node at (4.4,-0.5) {$n-2$};
    \node at (5.9,-0.5) {$n-1$};
\end{tikzpicture}
\]

\subsection{Superstable Configurations on $(P_n)_\bullet$}
We begin by counting the number of superstable configurations on $(P_n)_\bullet$. 

First, let $F_n$ denote the \textit{$n$th Fibonacci number}, given by $F_0=F_1=1$ and $F_n = F_{n-1} + F_{n-2}$. We begin with two identities of Fibonacci numbers.

\begin{lemma}
\label{lem:identities-of-fibonacci-numbers}
Let $F_n$ denote the $n$th Fibonacci number. Then, 
\begin{equation*}
    F_n = 2F_{n-1} - F_{n-3} \qquad \text{ and } \qquad F_n = 3F_{n-2}-F_{n-4}.
\end{equation*}
\end{lemma}
\begin{proof}
By definition, $F_{n-2} + F_{n-3} = F_{n-1}$ implies $F_{n-2} = F_{n-1} - F_{n-3}$. Combining these two equalities yields $F_n = 2F_{n-1} - F_{n-3}$. Notice that shifting indices yields $F_{n-1} = 2F_{n-2} - F_{n-4}$. Thus, $F_n = 3F_{n-2} - F_{n-4}$.
\end{proof}
Next, we work through some determinant computations.

\begin{lemma}
\label{lem:lemma-2}
Let $M_n$ denote the $n\times n$ matrix given by
\begin{equation*}
    \begin{bmatrix}
        3 & -1 & 0 & \cdots & 0 & 0 \\
        -1 & 3 & -1 & \cdots & 0 & 0 \\
        0 & -1 & 3 & \cdots & 0 & 0 \\
        \vdots & \vdots & \vdots & \ddots & \vdots & \vdots \\
        0 & 0 & 0 & \cdots & 3 & -1 \\
        0 & 0 & 0 & \cdots & -1 & 2
    \end{bmatrix}.
\end{equation*}
That is, $M_n$ has a 2 in the $n$th row, $n$th column, 3's along the diagonal in rows $1,\dots n-1$. It has $-1$'s along the super and subdiagonal, and 0's elsewhere. Then, $\det(M_n) = F_{2n+1}$ for any positive integer $n$.
\end{lemma}
\begin{proof}
The result follows by induction. The base case $\det(M_1) = F_3 = 2$ is immediate. Therefore, assume that $\det(M_k) = F_{2k+1}$ for all $k < n$. Then, by cofactor expansion along the top row, 
\begin{equation*}
   \det(M_n) = 3 \det(M_{n-1}) + \begin{vmatrix}
        -1 & -1 & \cdots & 0 & 0 \\
        0 & 3 & \cdots & 0 & 0 \\
        \vdots & \vdots & \ddots & \vdots & \vdots & \\
        0 & 0 & \cdots & 3 & -1 \\
        0 & 0 & \cdots & -1 & 2
    \end{vmatrix}=3 \det(M_{n-1})-\det(M_{n-2}).
\end{equation*}
Thus, $\det(M_n) = 3F_{2n-1} - F_{2n-3} = F_{2n+1}$ by Lemma \ref{lem:identities-of-fibonacci-numbers}, as desired.
\end{proof}

\begin{lemma}
\label{lem:lemma-3}
Let $M'_n$ denote the $n \times n$ matrix given by
\begin{equation*}
    \begin{bmatrix}
        2 & -1 & 0 & \cdots & 0 & 0 \\
        -1 & 3 & -1 & \cdots & 0 & 0 \\
        0 & -1 & 3 & \cdots & 0 & 0 \\
        \vdots & \vdots & \vdots & \ddots & \vdots & \vdots \\
        0 & 0 & 0 & \cdots & 3 & -1 \\
        0 & 0 & 0 & \cdots & -1 & 2
    \end{bmatrix}.
\end{equation*}
Then, $\det(M'_n) = F_{2n}$ for any positive integer $n$.
\end{lemma}
\begin{proof}
The result follows from cofactor expansion. Indeed, 
\begin{equation*}
    \det(M'_n) = 2 \begin{vmatrix}
        3 & -1 & \cdots & 0 & 0 \\
        -1 & 3 & \cdots & 0 & 0 \\
        \vdots & \vdots & \ddots & \vdots & \vdots \\
        0 & 0 & \cdots & 3 & -1 \\
        0 & 0 & \cdots & -1 & 2
    \end{vmatrix} + \begin{vmatrix}
        -1 & -1 & \cdots & 0 & 0 \\
        0 & 3 & \cdots & 0 & 0 \\
        \vdots & \vdots & \ddots & \vdots & \vdots & \\
        0 & 0 & \cdots & 3 & -1 \\
        0 & 0 & \cdots & -1 & 2
    \end{vmatrix}
\end{equation*}
where the first summand is equal to $2\det(M_{n-1})$, while the second is equal to $-\det(M_{n-2})$ from Lemma \ref{lem:lemma-2}. But Lemma \ref{lem:lemma-2} shows that $\det(M_{n-1}) = F_{2n-1}$ and $-\det(M_{n-2}) = -F_{2n-3}$. Thus, $\det(M'_n) = 2F_{2n-1} - F_{2n-3} = F_{2n}$ by Lemma \ref{lem:identities-of-fibonacci-numbers}.
\end{proof}

\begin{theorem}\label{thm:superstable-configurations-of-P_n}
The number of superstable configurations of $(P_n)_\bullet$ is $F_{2n}$ for any positive integer $n$.
\end{theorem}
\begin{proof}
Theorem 4.2.2 of \cite{Klivans} states that the number of superstable configurations of $(P_n)_\bullet$ is equal to the determinant of the reduced Laplacian $\Delta_q$ of $(P_n)_\bullet$ \cite{Klivans}. By definition, $\Delta_q((P_n)_\bullet) = M'_n$. Thus, by Lemma \ref{lem:lemma-3}, $\Delta_q((P_n)_\bullet)$ has determinant $F_{2n}$. 
\end{proof}

We refer the reader to Appendix \ref{subsec:app-b-the-path-graph} for a direct combinatorial proof.

Next, we offer a characterization of superstable configurations on $(P_n)_\bullet$.

\begin{definition}[$2$-Free Block]\label{def:2-free-block}
Given a sequence of integers ${\bf a} = (a_1,\dots,a_n)\in\{0,1,2\}^n$, we call a maximal contiguous subsequence consisting of $0$'s and $1$'s that occur to the left or right of a $2$ a \textit{$2$-free block}. Note that we also consider empty subsequences as 2-free blocks. Thus, if there are $m-1$ occurrences of $2$ in ${\bf a}$, then ${\bf a}$ has $m$ $2$-free blocks. 
\end{definition}

\begin{example}\label{ex:2-free-block}
If ${\bf a} = (2,0,1,2,2,0,2)$, then the $2$-free blocks are $\pi_1 = ()$, $\pi_2 = (0,1)$, $\pi_3 = ()$, $\pi_4 = (0)$, and $\pi_5=()$.
\end{example}

\begin{lemma}[Superstability Criterion]\label{lem:Carls-criterion}
Suppose ${\bf a} \in \{0,1,2\}^n$ has 2-free blocks $\pi_1,\pi_2,\cdots,\pi_m$. Then ${\bf a}$ is a superstable configuration on $(P_n)_\bullet$ if and only if there is a $0$ in each 2-free block $\pi_i$.
\end{lemma}
\begin{proof}
For necessity, note that if any 2-free block $\pi_i$ contains only $1$s, then firing at $\pi_i$ as well as its bordering $2$s is a legal cluster-fire. For sufficiency, we use Dhar's Burning Algorithm. Indeed, if the Superstability Criterion holds, the fire spreads to each vertex with $0$ chips, and any vertex in any of the $2$-free blocks is eventually on fire. By this point, any of the vertices with $2$ chips have $3$ incident burning edges, so they light on fire too. Therefore, by Theorem \ref{def:dhars-burning-algorithm-detects-superstability}, ${\bf a}$ is superstable.
\end{proof}

\subsection{The Three Rows Game}\label{sec:three-rows-game}
Next, we define the Three Rows Game, and explain why it encodes the labelling of the regions of the $G$-Shi arrangement according to the Pak-Stanley algorithm.

\begin{definition}[Three Rows Game]\label{def:the-three-rows-game}
The \textit{Three Rows Game} is played on the following board: 
\begin{table}[H]
    \centering
    \begin{tabular}{|l|l|l|l|l|l|l|}
    \hline
    0 & 1 & 2 & 3 & $\cdots$ & $n-3$ & $n-2$ \\ \hline
     &  &  &  &  &  &  \\ \hline
    1 & 2 & 3 & 4 & $\cdots$ & $n-2$ & $n-1$ \\ \hline
    \end{tabular}
    \label{tab:the-board-of-the-three-rows-game}
\end{table}
The game is played by choosing an entry in each column. The \textit{history} ${\bf h}$ of the game is the sequence of squares chosen (top, middle, or bottom), and the \textit{outcome} ${\bf o}({\bf h})$ or ${\bf o}$ of the game is the multiset of numbers chosen, ignoring any choice of middle squares. 
\end{definition}

In this section, our principal object of study is the \textit{multiplicity} of an outcome of the game.
\begin{definition}[Equivalent Histories]\label{def:equivalent-histories}
Two histories $\textbf{h}$ and $\textbf{h}'$ are \textit{equivalent}, denoted $\textbf{h}\sim {\bf h}'$ if ${\bf o}(\textbf{h})={\bf o}(\textbf{h}')$. This gives an equivalence relation, the equivalence classes of which are denoted $[{\bf h}]$.
\end{definition}
\begin{definition}[Multiplicity of an Outcome]\label{def:repetition-of-an-outcome}
Given an outcome ${\bf o}$, the \textit{multiplicity $\mu({\bf o})$ of ${\bf o}$} is the number of histories ${\bf h}$ such that ${\bf o}({\bf h}) = {\bf o}$. Furthermore, we adopt the convention that $\mu({\bf h})$ denotes $\mu({\bf o}({\bf h}))$. Notice that $\mu({\bf h})$ is the size of the equivalence class $[{\bf h}]$.
\end{definition}

\begin{example}\label{ex:the-three-rows-game}
Consider the following equivalent histories ${\bf h}\sim {\bf h'}$ of the Three Rows Game on 4 columns with ${\bf o}({\bf h})=\{0,2,4\}={\bf o}({\bf h'})$.
\begin{center}
    ${\bf h} =$ \begin{tabular}{|l|l|l|l|}
    \hline
    \cellcolor[HTML]{C0C0C0}0 & 1 & 2 & 3 \\ \hline
    &  & \cellcolor[HTML]{C0C0C0} &  \\ \hline
    1 & \cellcolor[HTML]{C0C0C0}2 & 3 & \cellcolor[HTML]{C0C0C0}4 \\ \hline
    \end{tabular}
    \quad \text{and} \quad
    ${\bf h}' =$ \begin{tabular}{|l|l|l|l|}
    \hline
    \cellcolor[HTML]{C0C0C0}0 & 1 & \cellcolor[HTML]{C0C0C0}2 & 3 \\ \hline & \cellcolor[HTML]{C0C0C0} &  &  \\ \hline
    1 & 2 & 3 & \cellcolor[HTML]{C0C0C0}4 \\ \hline
    \end{tabular}
\end{center}
Moreover, one can verify that these are the only two histories that yield outcome ${\bf o}=\{0,2,4\}$ so that $\mu({\bf o})=2$.
\end{example}

The reason why this game and the multiplicities of its outcomes are worth studying is the following correspondence. 

\begin{figure}[H]
    \centering
    \begin{tabular}{|l|}
    \hline
    Histories of the Three Rows Game \\ \hline
    Outcomes of the Three Rows Game \\ \hline
    \end{tabular} \begin{tabular}{l}
    $\Longleftrightarrow$ \\
    $\Longleftrightarrow$
    \end{tabular} \begin{tabular}{|l|}
    \hline
    Regions/chambers in the $G$-Shi arrangement \\ \hline
    Pak-Stanley labels in the $G$-Shi arrangement \\ \hline
    \end{tabular}
    \label{tab:correspondence-of-the-three-rows-game-and-G-Shi-arrangement}
\end{figure}
In particular, repetitions of Pak-Stanley labels in the Shi adjacency digraph correspond to instances where multiple histories of the Three Rows Game induce the same outcome.

To understand this correspondence, recall that each edge $\{i,i+1\}$ in the path graph $P_n$ gives a pair of parallel hyperplanes $x_i - x_{i+1} = 0$ and $x_i - x_{i+1} = 1$. Each region is uniquely determined by whether or not it lies ``before" both of these hyperplanes (it satisfies $x_i - x_{i+1} < 0$), between these hyperplanes (it satisfies $0 < x_i - x_{i+1} < 1$), or ``after" these hyperplanes (it satisfies $x_i - x_{i+1} > 1$), for $i = 0,\dots,n-2$. Therefore, choosing a region corresponds to choosing a square in each column of the following table.
\begin{table}[H]
\centering
\begin{tabular}{|c|c|c|c|c|}
\hline
$x_0-x_1<0$ & $x_1-x_2<0$ & $\cdots$ & $x_{n-3}-x_{n-2}<0$ & $x_{n-2}-x_{n-1}<0$ \\ \hline
$0<x_0-x_1<1$ & $0<x_1-x_2<1$ & $\cdots$ & $0<x_{n-3}-x_{n-2}<1$ & $0<x_{n-2}-x_{n-1}<1$ \\ \hline
$x_0-x_1>1$ & $x_1-x_2>1$ & $\cdots$ & $x_{n-3}-x_{n-2}>1$ & $x_{n-2}-x_{n-1}>1$ \\ \hline
\end{tabular}
\label{tab:from-regions-to-rows}
\end{table}

However, notice that not all of this information is necessary to compute the Pak-Stanley label. Indeed, choosing $x_0 - x_1 < 0$ corresponds to crossing the hyperplane $x_0 - x_1 = 0$, which means incrementing the label in coordinate $0$. Similarly, choosing $x_0 - x_1 > 1$ means incrementing the label in coordinate $1$, and choosing $0 < x_0 - x_1 < 1$ means that no hyperplanes are crossed, and therefore none of the coordinates of the label are incremented. Therefore, we replace each inequality with the coordinate of the label that is incremented, which yields the familiar board of the Three Rows Game in Definition \ref{def:the-three-rows-game}. 

Then, to compute the Pak-Stanley label corresponding to a particular outcome ${\bf o}$, we start with $(0,\dots,0)$ and increment the $i$th coordinate for each occurence of $i$ in ${\bf o}$. 

\begin{example} 
Continuing Example \ref{ex:the-three-rows-game}, consider playing the Three Rows Game on 4 columns, and receiving the outcome $\{0,2,4\}$. This corresponds to crossing edges labelled with a $0$, a $2$, and a $4$, and therefore incrementing each of the coordinates $0$, $2$, and $4$ once to get the $G$-parking function $10101$.
\end{example}

This exact correspondence of regions and histories, labels and outcomes, shows that the Three Rows Game is a concise combinatorial model that captures the complexity of the behavior of the Pak-Stanley algorithm on path graphs. With the language of the Three Rows Game, we can more easily spot patterns in the behavior of the Pak-Stanley algorithm. 

An example of one such pattern is that as we play the Three Rows Game, the outcomes are independent of any permutation of the columns. This demonstrates the following fact:


\begin{proposition}
The multiplicities of the Pak-Stanley labels in the Shi adjacency digraph of $P_n$ are independent of the vertex labelling up to permutation. 
\end{proposition}

\subsection{Uniqueness of Maximal Labels in the $P_n$-Shi Arrangement}
In this section we study the multiplicities of maximal labels of regions of the $P_n$-Shi arrangement. 



\begin{theorem}\label{thm:maximal-sink-history-without-middle}
The following are equivalent for a vertex $v$ in the Shi adjacency digraph of $P_n$: 
\begin{enumerate}
    \item $v$ is a sink.
    \item $v$ is labelled with a maximal $G_\bullet$-parking function.
    \item The history corresponding to $v$ has no entries from the middle row of the Three Rows Game.
\end{enumerate}
\end{theorem}
\begin{proof}
\textbf{(3) $\Rightarrow$ (2): }In this case, the Pak-Stanley label ${\bf p}$ of $v$ has coordinate sum $n-1$, corresponding to a superstable configuration with $n-1$ total chips: one for each column in the Three Rows Game. Thus, by Proposition \ref{prop:maximal-superstable-conf} and Theorem \ref{thm:superstable-configurations-are-G-parking-functions}, ${\bf p}$ is maximal.

\textbf{(2) $\Rightarrow$ (1): }If $v$ is labelled with a maximal $G_\bullet$-parking function, then $v$ must be a sink, as otherwise any children of $v$ would be labelled with strictly larger $G_\bullet$-parking functions, a contradiction.

\textbf{(1) $\Rightarrow$ (3): }Suppose that the history corresponding to $v$ has a middle square in column $i$. This implies that $v$ satisfies $0 < x_i - x_{i+1} < 1$, but as discussed in Section \ref{sec:SAD-of-trees}, any region between two hyperplanes has outdegree at least $2$. It is the parent of the two regions given by crossing the hyperplanes $x_i - x_{i+1} = 0$ and $x_i - x_{i+1} = 1$, respectively), so it cannot be a sink.
\end{proof}



This discussion immediately demonstrates the following fact: 

\begin{proposition}
There are $2^{n-1}$ sinks in the Shi adjacency digraph of $P_n$.
\end{proposition}
\begin{theorem}[Uniqueness of Maximal Outcomes]\label{thm:uniqueness-of-maximal-outcomes}
Any outcome ${\bf o}$ of the Three Rows Game on $n$ columns with $|{\bf o}| = n$ has multiplicity $\mu({\bf o}) = 1$.
\end{theorem}
\begin{proof}
Suppose ${\bf h}$ and ${\bf h}'$ are two equivalent histories whose shared outcome has cardinality $n$. Then ${\bf h}$ and ${\bf h}'$ are both length $n$ sequences of squares with labels in weakly increasing order, which induce the same outcome if and only if they are equal, ${\bf h} = {\bf h}'$.
\end{proof}
\begin{corollary}
Any maximal $(P_n)_\bullet$-parking function appears uniquely as a Pak-Stanley label for $P_n$. In other words, the labels on all of the sinks of the $P_n$-Shi adjacency digraph appear uniquely.
\end{corollary}

Now, while uniqueness is straightforward to prove for maximal outcomes, the following example shows that it fails quite drastically for general outcomes. 

\begin{example}\label{ex:exponentially-many-repetitions}
Consider the Three Rows Game on $2n$ columns. Then the outcome ${\bf o} = \{1,3,5,\dots,2n-1\}$ satisfies $\mu({\bf o}) = 2^n$. To see why, notice that there are $n$ integers in this outcome, and for each of them we can choose either $i\square$ (bottom and middle) or $\square i$ (middle and top).
\end{example}
This example is interesting because it shows that some outcomes of the Three Rows Game can arise from exponentially many histories. Later, once we develop further machinery, we will find the outcome (and therefore the Pak-Stanley label) which occurs the maximum number of times.

\subsection{Patterns in the Three Rows Game}
To characterize why uniqueness fails for general outcomes, we introduce the notion of patterns, which are particular sequences of moves in the Three Rows Game.

\begin{definition}[Ascents/Descents]\label{def:ascents/descents}
Given a history ${\bf h} = (h_1,\dots,h_n)$ of the Three Rows Game, $(h_i,h_{i+1})$ is called an \textit{ascent} if $h_{i+1}$ is in a higher row than $h_i$, and a \textit{descent} if $h_{i+1}$ is in a lower row than $h_i$. Any subsequence $(h_i,h_{i+1},\dots,h_{j-1},h_j)$ of ${\bf h}$ is called \textit{weakly ascending (resp. descending)} if it has no descent (resp. ascent).
\end{definition}

Note that ascents and descents in a history are defined solely by the \textit{placement} of the blocks, rather than the values that they contain.

\begin{definition}[Prepatterns and Patterns]\label{def:(pre)patterns}
Given a history of the Three Rows Game, we split the history into blocks, called \textit{prepatterns}, by introducing breaks before any descent or before a middle square appears twice in a block. Those prepatterns which contain a middle square are called \textit{patterns}. The length of a pattern ${\bf p}$ is denoted $\ell({\bf p})$.
\end{definition}

\begin{example}
The following history is broken up into its 6 prepatterns: 
\begin{figure}[H]
    \centering
    \begin{tabular}{|l|l|l|l|l|l|l|l|l|l|}
    \hline
    \cellcolor[HTML]{C0C0C0}0 & 1 & 2 & 3 & 4 & \cellcolor[HTML]{C0C0C0}5 & 6 & 7 & 8 & 9 \\ \hline
     & \cellcolor[HTML]{C0C0C0} &  &  & \cellcolor[HTML]{C0C0C0} &  &  & \cellcolor[HTML]{C0C0C0} & \cellcolor[HTML]{C0C0C0} &  \\ \hline
    1 & 2 & \cellcolor[HTML]{C0C0C0}3 & \cellcolor[HTML]{C0C0C0}4 & 5 & 6 & \cellcolor[HTML]{C0C0C0}7 & 8 & 9 & \cellcolor[HTML]{C0C0C0}10 \\ \hline
    \end{tabular} 
    $\quad \Rightarrow \quad$ \begin{tabular}{|l|}
    \hline
    \rowcolor[HTML]{C0C0C0} 
    0 \\ \hline
     \\ \hline
    1 \\ \hline
    \end{tabular} \quad 
    \begin{tabular}{|l|}
    \hline
    1 \\ \hline
    \rowcolor[HTML]{C0C0C0} 
     \\ \hline
    2 \\ \hline
    \end{tabular} \quad 
    \begin{tabular}{|l|l|l|l|}
    \hline
    2 & 3 & 4 & \cellcolor[HTML]{C0C0C0}5 \\ \hline
     &  & \cellcolor[HTML]{C0C0C0} &  \\ \hline
    \cellcolor[HTML]{C0C0C0}3 & \cellcolor[HTML]{C0C0C0}4 & 5 & 6 \\ \hline
    \end{tabular} \quad 
    \begin{tabular}{|l|l|}
    \hline
    6 & 7 \\ \hline
     & \cellcolor[HTML]{C0C0C0} \\ \hline
    \cellcolor[HTML]{C0C0C0}7 & 8 \\ \hline
    \end{tabular} \quad 
    \begin{tabular}{|l|}
    \hline
    $8$ \\ \hline
    \rowcolor[HTML]{C0C0C0} 
     \\ \hline
    $9$ \\ \hline
    \end{tabular} \quad \begin{tabular}{|l|}
    \hline
    $9$ \\ \hline
    \rowcolor[HTML]{FFFFFF} 
     \\ \hline
    \rowcolor[HTML]{C0C0C0} 
    $10$ \\ \hline
    \end{tabular}
    \label{fig:splitting-a-history-into-prepatterns}
\end{figure}
From those 6 prepatterns, we get the following patterns whose lengths are 1, 4, 2, and 1 respectively: 
\begin{figure}[H]
    \centering
    \begin{tabular}{|l|}
    \hline
    1 \\ \hline
    \rowcolor[HTML]{C0C0C0} 
     \\ \hline
    2 \\ \hline
    \end{tabular} \quad 
    \begin{tabular}{|l|l|l|l|}
    \hline
    2 & 3 & 4 & \cellcolor[HTML]{C0C0C0}5 \\ \hline
     &  & \cellcolor[HTML]{C0C0C0} &  \\ \hline
    \cellcolor[HTML]{C0C0C0}3 & \cellcolor[HTML]{C0C0C0}4 & 5 & 6 \\ \hline
    \end{tabular} \quad 
    \begin{tabular}{|l|l|}
    \hline
    6 & 7 \\ \hline
     & \cellcolor[HTML]{C0C0C0} \\ \hline
    \cellcolor[HTML]{C0C0C0}7 & 8 \\ \hline
    \end{tabular} \quad \begin{tabular}{|l|}
    \hline
    $8$ \\ \hline
    \rowcolor[HTML]{C0C0C0} 
     \\ \hline
    $9$ \\ \hline
    \end{tabular}
    \label{fig:remaining-patterns}
\end{figure}
\end{example}

The importance of patterns comes when proving Theorem  \ref{thm:pattern-repetition-theorem}. The proof of this theorem requires uniqueness of maximal outcomes as well as the following series of technical lemmas in Section \ref{subsec:technical-lemmas}. Furthermore, these technical lemmas will also prove useful in translating our results back into the language of Pak-Stanley labels.

\subsection{Technical Lemmas About Patterns}\label{subsec:technical-lemmas}

\begin{lemma}[Gap Characterization Lemma]\label{lem:gap-characterization-lemma}
Suppose that ${\bf h} = (h_1,\dots,h_n)$ is a history of the Three Rows Game with outcome ${\bf o}$ and $i\in\{0,1,\dots,n\}$. Then the following are equivalent: 
\begin{enumerate}
    \item $i$ does not appear in the outcome ${\bf o}$.
    \item Either $(h_i,h_{i+1})$ is a descent or both $h_i$ and $h_{i+1}$ are middle squares.
\end{enumerate}
\end{lemma}
\begin{proof}The only two places that $i$ appears is at the bottom column $i$ and the top of column $(i+1)$. Therefore, $i$ does not appear in ${\bf o}$ if and only if $h_i$ is not the bottom square and $h_{i+1}$ is not the top square; in this case, either $(h_i,h_{i+1})$ is a descent or both $h_i$ and $h_{i+1}$ are middle squares. 
\end{proof}

\begin{lemma}[Splitting Lemma]\label{lem:splitting-lemma}
Suppose that ${\bf h} = (h_1,\dots,h_n)$ and ${\bf h}' = (h'_1,\dots,h'_n)$ are histories of the Three Rows Game where ${\bf h}$ has outcome ${\bf o}$ with $i \not\in {\bf o}$. Then ${\bf h} \sim {\bf h'}$ if and only if ${\bf h} \sim (h'_1,\dots,h'_i,h_{i+1},\dots,h_n)\sim (h_1,\dots,h_i,h_{i+1}',\dots,h_n')$. That is, any replacement of ${\bf h}$ by an equivalent history can be broken up into two independent replacements of $(h_1,\dots,h_i)$ and $(h_{i+1},\dots,h_n)$ by individually equivalent histories, so that
\begin{equation*}
    \mu((h_1,\dots,h_n)) = \mu((h_1,\dots,h_i)) \cdot \mu((h_{i+1},\dots,h_n)).
\end{equation*}
\end{lemma}
\begin{proof}
Suppose that $i \not\in {\bf o}$. Then, ${\bf o}$ can be split into two parts ${\bf o}_1$ and ${\bf o}_2$: those numbers which are strictly less than $i$, and those numbers which are strictly greater than $i$. Hence, a replacement of $h_1,\dots,h_i$ by $h_1',\dots,h_i'$ can only affect ${\bf o}_1$, and a replacement of $h_{i+1},\dots,h_n$ by $h_{i+1}',\dots,h_n'$ can only affect ${\bf o}_2$. The result follows.
\end{proof}

\begin{lemma}[Pattern Characterization Lemma]\label{lem:pattern-characterization-lemma}
Suppose that ${\bf h} = (h_1,\dots,h_n)$ is a history of the Three Rows Game with outcome ${\bf o}$. Then, if $1 \leq i < j \leq n$, $(h_i,h_{i+1},\dots,h_j)$ is a pattern if and only if $i,i+1,\dots,j-1$ appear exactly once in ${\bf o}$ and $i-1,j \not\in {\bf o}$.
\end{lemma}
\begin{proof}
Suppose that $(h_i,h_{i+1},\dots,h_j)$ is a pattern. Then $(h_i,h_{i+1},\dots,h_j)$ takes the form:
\begin{table}[H]
    \centering
    \begin{tabular}{lllllll}
    \hline
    \multicolumn{1}{|l|}{$i-1$} & \multicolumn{1}{l|}{} & \multicolumn{1}{l|}{} & \multicolumn{1}{l|}{} & \multicolumn{1}{l|}{\cellcolor[HTML]{C0C0C0}$\;\,$} & \multicolumn{1}{l|}{\cellcolor[HTML]{C0C0C0}$\cdots$} & \multicolumn{1}{l|}{\cellcolor[HTML]{C0C0C0}$j-1$} \\ \hline
    \multicolumn{1}{|l|}{} & \multicolumn{1}{l|}{} & \multicolumn{1}{l|}{} & \multicolumn{1}{l|}{\cellcolor[HTML]{C0C0C0}$\;\,$} & \multicolumn{1}{l|}{} & \multicolumn{1}{l|}{} & \multicolumn{1}{l|}{} \\ \hline
    \multicolumn{1}{|l|}{\cellcolor[HTML]{C0C0C0}$i$} & \multicolumn{1}{l|}{\cellcolor[HTML]{C0C0C0}$\cdots$} & \multicolumn{1}{l|}{\cellcolor[HTML]{C0C0C0}$\;\,$} & \multicolumn{1}{l|}{} & \multicolumn{1}{l|}{} & \multicolumn{1}{l|}{} & \multicolumn{1}{l|}{$j$} \\ \hline
    $h_i$ & \multicolumn{2}{l}{} & $h_{i+k}$ & \multicolumn{2}{l}{} & $h_j$
    \end{tabular}
    \label{tab:structure-of-a-pattern}
\end{table}
That is, $(h_i,h_{i+1},\dots,h_j)$ begins with $k$ bottom squares, a single middle square, and then $j-i-k$ top squares for some $k \in \{0,\dots,j-i\}$. This establishes that $i,i+1,\dots,j-1$ appear exactly once in ${\bf o}$. Then, since $h_{i-1}$ is not the bottom square (as otherwise it would be in the pattern), $i-1 \not\in {\bf o}$. Similarly, since $h_{j+1}$ is not the top square, $j \not\in {\bf o}$.

Conversely, suppose $\{i,i+1,\dots,j-1\}\subseteq {\bf o}$ and $i-1,j \not\in {\bf o}$. Then, since $j-i-1$ numbers appear from the $j-i$ columns $h_i,h_{i+1},\dots,h_j$, there is one middle square among them. Furthermore, there is no descent in $(h_i,h_{i+1},\dots,h_j)$ by Lemma \ref{lem:gap-characterization-lemma}. Now, since $i-1 \not\in {\bf o}$, Lemma \ref{lem:gap-characterization-lemma} shows that $(h_{i-1},h_i)$ is either a descent or two middle squares. The same holds for $(h_j,h_{j+1})$ as $j \not\in {\bf o}$. Thus, $(h_i,\dots,h_j)$ is a pattern: it is weakly ascending with exactly one middle square, and $h_{i-1}$ and $h_{j+1}$ cannot be included.
\end{proof}

\begin{lemma}[Multiplicity of Patterns]\label{lem:repetitions-of-patterns}
Suppose that ${\bf h} = (h_1,\dots,h_n)$ is a history of the Three Rows Game which is also a pattern (that is, it is weakly ascending with exactly one middle square). Then $\mu({\bf h}) = n$, and all the histories in $[{\bf h}]$ are also patterns.
\end{lemma}
\begin{proof}
First, by Lemma \ref{lem:pattern-characterization-lemma}, ${\bf h}' \sim {\bf h}$ if and only if ${\bf h}'$ is also a pattern. Now, as previously discussed, patterns take the following form: 
\begin{table}[H]
    \centering
    \begin{tabular}{|l|l|l|l|l|l|l|}
    \hline
     &  &  &  & \cellcolor[HTML]{C0C0C0}$\;\,$ & \cellcolor[HTML]{C0C0C0}$\cdots$ & \cellcolor[HTML]{C0C0C0}$\;\,$ \\ \hline
     &  &  & \cellcolor[HTML]{C0C0C0}$\;\,$ &  &  &  \\ \hline
    \cellcolor[HTML]{C0C0C0}$\;\,$ & \cellcolor[HTML]{C0C0C0}$\cdots$ & \cellcolor[HTML]{C0C0C0}$\;\,$ &  &  &  &  \\ \hline
    \end{tabular}
    \label{tab:structure-of-a-pattern-2}
\end{table}
That is, the first $k \in \{0,1,\dots,n-1\}$ squares are chosen in the bottom row, the next square is chosen in the middle row, and the remaining squares are chosen in the top row. Since there are $n$ options for $k$, this demonstrates that there are $n$ possible patterns of length $n$, as desired.
\end{proof}

\subsection{The Pattern Multiplicity Theorem and Consequences}
\begin{theorem}[Pattern Multiplicity Theorem]\label{thm:pattern-repetition-theorem}
Suppose that ${\bf h}$ is a history of the Three Rows Game with patterns ${\bf p}_1,\dots,{\bf p}_k$. Then $\mu({\bf h}) = \ell({\bf p}_1) \cdots \ell({\bf p}_k)$. 
\end{theorem}
\begin{proof}
First, notice that by Lemma \ref{lem:gap-characterization-lemma} and \ref{lem:splitting-lemma}, if ${\bf p}'_1,\dots,{\bf p}'_{k'}$ are the prepatterns of ${\bf h}$, then $\mu({\bf h}) = \mu({\bf p}'_1)\cdots \mu({\bf p}'_{k'})$. Now, if ${\bf p}'_i$ is not a pattern, then it has no middle squares. But then by Theorem \ref{thm:uniqueness-of-maximal-outcomes} (Uniqueness of Maximal Outcomes), $\mu({\bf p}'_i) = 1$. In other words, we can drop all the terms of the product corresponding to prepatterns which are not patterns; this demonstrates that $\mu({\bf h}) = \mu({\bf p}_1)\cdots \mu({\bf p}_k)$. On the other hand, if ${\bf p}_i$ is a pattern, then by Lemma \ref{lem:repetitions-of-patterns}, $\mu({\bf p}_i) = \ell({\bf p}_i)$. Therefore, $\mu({\bf h}) = \ell({\bf p}_1) \cdots \ell({\bf p}_k)$. 
\end{proof}

Now we will restate our work in terms of Pak-Stanley labellings. The key here is that patterns take the form $(0,1,\dots,1,0)$ in the Pak-Stanley label. 

\begin{corollary}\label{cor:outcomes-and-run-repetitions}
Suppose that ${\bf o}$ is an outcome. A \textit{run} ${\bf r}$ of length $k$ in ${\bf o}$ is a sequence of $k$ consecutive positive integers $i,i+1,\dots,i+k-1 \in [n-1]$ each appearing exactly once in ${\bf o}$ such that $i-1,i+k \not\in {\bf o}$. If the length of a run ${\bf r}$ is denoted $\ell({\bf r})$, the number of histories inducing ${\bf o}$ is 
\begin{equation*}
    \mu({\bf o}) = \prod_{\text{runs ${\bf r}$ in ${\bf o}$}} (\ell({\bf r})+1).
\end{equation*}
\end{corollary}
\begin{proof}
This follows directly from Theorem \ref{thm:pattern-repetition-theorem} and Lemma \ref{lem:pattern-characterization-lemma}.
\end{proof}

\begin{theorem}[Path Multiplicity Theorem]\label{thm:path-repetition-theorem}
Suppose ${\bf p} = (p_1,\dots,p_n)$ is a
Pak-Stanley label on $P_n$. A \textit{run} ${\bf r}$ of length $n$ in ${\bf p}$ is a section of ${\bf p}$ of the form $(0,1,\dots,1,0)$ with $n$ $1$s. If the length of a run ${\bf r}$ is denoted $\ell({\bf r})$, then the multiplicity of the label ${\bf p}$ in in the $P_n$-Shi arrangement is
\begin{equation*}
    \mu({\bf p}) = \prod_{\text{runs ${\bf r}$ in ${\bf p}$}} (\ell({\bf r})+1).
\end{equation*}
    \end{theorem}
\begin{proof}
This follows directly from Corollary \ref{cor:outcomes-and-run-repetitions} and the connection between histories and Pak-Stanley labellings; namely, an outcome ${\bf o}$ corresponds to the Pak-Stanley label ${\bf p}$ given by letting coordinate $i$ of the Pak-Stanley label equal the number of occurences of $i$ in ${\bf o}$.
\end{proof}

Using the Path Multiplicity Theorem, we can further explore the structure of the Pak-Stanley labels of the $P_n$-Shi arrangement. Indeed, we give two opposite results: a characterization of all the labels which appear only once, and a characterization of the labels which appear the maximum number of times (and how many times they appear). The former follows from the Path Multiplicity Theorem (Theorem \ref{thm:path-repetition-theorem}). 

\begin{corollary}\label{cor:characterization-of-unique-labels-on-P_n}
Suppose ${\bf p} = (p_1,\dots,p_n)$ is a $(P_n)_\bullet$-parking function. Then ${\bf p}$ appears uniquely as a Pak-Stanley label for $P_n$ if and only if, for every nonconsecutive pair of $0$s in ${\bf p}$, there is a $2$ between the $0$s.
\end{corollary}

The latter is a little more difficult, and involves the following optimization problem: if $n_1,\dots,n_k$ are positive integers which add up to $n$, what is the largest possible value of $n_1\cdots n_k$?

\begin{lemma}\label{lem:technical-lemma-optimization}
Suppose that $n > 1$ is a positive integer. Let
\begin{equation*}
    S(n) = \sup_{\substack{n_1,\dots,n_k \in \Z^+ \\ n_1 + \cdots + n_k = n}} n_1 \cdots n_k.
\end{equation*}
Then, 
\begin{enumerate}
    \item if $n \equiv 0 \bmod 3$, then $S(n) = 3^{n/3}$,
    \item if $n \equiv 1 \bmod 3$, then $S(n) = 3^{(n-4)/3} \cdot 2^2$,
    \item if $n \equiv 2 \bmod 3$, then $S(n) = 3^{(n-2)/3} \cdot 2$.
\end{enumerate}
\end{lemma}
\begin{proof}
If $n_i \geq 4$, then $n_i$ can be replaced by $n_i - 2$ and $2$, which have product $2n_i - 4 \geq n_i$. Thus, we may assume that $n_i \in \{2,3\}$ for each $i$. But $3 + 3 = 2 + 2 + 2$ yet $3^2 > 2^3$, so there will never be three $2$s among the $n_i$. The result then follows in each case by inspection.
\end{proof}

Using this optimization result, we can compute the Pak-Stanley labels which appear the most often.

\begin{proposition}\label{prop:maximally-occuring-pak-stanley-labels}
Suppose that $n > 2$. Then, the following holds.
\begin{enumerate}
    \item If $n \equiv 1 \bmod 3$, then the unique Pak-Stanley label for $P_n$ which appears most often is $0110110 \cdots 110$, and it is repeated $3^{(n-1)/3}$ times.
    \item If $n \equiv 2 \bmod 3$, then any of the ${(n+1)/3 \choose 2}$ Pak-Stanley labels for $P_n$ starting with a $0$ and then containing $(n-5)/3$ copies of $110$ and $2$ copies of $10$ (for example, $01010110 \cdots 110$) in any order appear most often, and they are repeated $2^2 \cdot 3^{(n-5)/3}$ times.
    \item If $n \equiv 0 \bmod 3$, then any of the $\frac{n}{3}$ Pak-Stanley labels for $P_n$ starting with a $0$ and then containing $(n-3)/3$ copies of $110$ and $1$ copy of $10$ (for example, $010110 \cdots 110$) in any order appear most often, and they are repeated $2 \cdot 3^{(n-3)/3}$ times.
\end{enumerate}
\end{proposition}
\begin{proof}
This result follows from the fact that repetitions are created by patterns, so to achieve a label which has maximum multiplicity, we should construct a label made entirely out of patterns. Such a label starts with a $0$ which is followed by a number of runs of the form $1 \cdots 10$. The total length of all of these runs (including the $0$s) is $n-1$, and the product of their lengths (including the final $0$ in each one) is the desired quantity. Therefore, this problem reduces to the problem in Lemma \ref{lem:technical-lemma-optimization} of finding $S(n-1)$.
\end{proof}

Notice that not only do we know the number of vertices with the label ${\bf p}$, we also know how to list all vertices with label ${\bf p}$. This gives us the ability to easily iterate over all such vertices if we so desire. We discuss this further in the next section.

\subsection{Iterating Over Regions with a Fixed Label}
In this section, we discuss two related problems: first, the problem of listing all regions with a fixed label, and second, the problem of finding where the regions with a fixed label are located.

\begin{question}
Given a Pak-Stanley label ${\bf p}$ for $P_n$, how can we list the regions of the $P_n$-Shi arrangement with label ${\bf p}$?
\end{question}
\begin{example}\label{ex:pak-to-hitory}
Given the Pak-Stanley label of a certain region, we may reconstruct the history of the Three Rows Game to determine which choices are fixed. For example, given the Pak-Stanley label $(0,1,1,0,2,0)$, we may determine that the outcome was $\{1,2,4,4\}$. Since $4$ only appears twice in the Three Rows Game, and appears twice in the outcome, the following two columns are fixed:
\begin{table}[H]
    \centering
    \begin{tabular}{|c|c|c|c|c|}
    \hline
    0 & 1 & 2 & 3 & \cellcolor[HTML]{C0C0C0}4 \\ \hline
     &  &  &  &  \\ \hline
    1 & 2 & 3 & \cellcolor[HTML]{C0C0C0}4 & 5 \\ \hline
    \end{tabular}
    \label{tab:fixing-columns-in-the-three-rows-game}
\end{table}

Now, notice that $\{1,2\}$ is a run in the sense of Corollary \ref{cor:outcomes-and-run-repetitions}, and therefore corresponds to a pattern of length $3$. This yields the following three options for the first three rows: 

\begin{figure}[H]
    \centering
    \begin{tabular}{|c|c|c|c|c|}
        \hline
        0 & \cellcolor[HTML]{C0C0C0}1 & \cellcolor[HTML]{C0C0C0}2 & 3 & \cellcolor[HTML]{C0C0C0}4 \\ \hline
        \cellcolor[HTML]{C0C0C0} &  &  &  &  \\ \hline
        1 & 2 & 3 & \cellcolor[HTML]{C0C0C0}4 & 5 \\ \hline
        \end{tabular}\qquad
        \begin{tabular}{|
        >{\columncolor[HTML]{FFFFFF}}c |c|c|c|c|}
        \hline
        0 & \cellcolor[HTML]{FFFFFF}1 & \cellcolor[HTML]{C0C0C0}2 & 3 & \cellcolor[HTML]{C0C0C0}4 \\ \hline
         & \cellcolor[HTML]{C0C0C0} &  &  &  \\ \hline
        \cellcolor[HTML]{C0C0C0}1 & 2 & 3 & \cellcolor[HTML]{C0C0C0}4 & 5 \\ \hline
        \end{tabular}\qquad
        \begin{tabular}{|
        >{\columncolor[HTML]{FFFFFF}}c |
        >{\columncolor[HTML]{FFFFFF}}c |c|c|c|}
        \hline
        0 & 1 & \cellcolor[HTML]{FFFFFF}2 & 3 & \cellcolor[HTML]{C0C0C0}4 \\ \hline
         &  & \cellcolor[HTML]{C0C0C0}{\color[HTML]{C0C0C0} } &  &  \\ \hline
        \cellcolor[HTML]{C0C0C0}1 & \cellcolor[HTML]{C0C0C0}2 & 3 & \cellcolor[HTML]{C0C0C0}4 & 5 \\ \hline
    \end{tabular}
    \caption{Enumerating possible histories for the outcome $\{1,2,4,4\}$.}
    \label{fig:all-possibilities}
\end{figure}
\end{example}

This example illustrates the general strategy. First, convert the Pak-Stanley label to an outcome of the Three Rows Game. Then, consider all the prepatterns which are not patterns; as discussed in the proof of Theorem \ref{thm:pattern-repetition-theorem}, these have only one option, so we fix this option. Then, for each of the remaining patterns ${\bf p}$, there are $\ell({\bf p})$ options of the following form: 
\begin{table}[H]
    \centering
    \begin{tabular}{|l|l|l|l|}
    \hline
    \rowcolor[HTML]{C0C0C0} 
    \cellcolor[HTML]{FFFFFF} & $\;\,$ & $\cdots$ & $\;\,$ \\ \hline
    \cellcolor[HTML]{C0C0C0}$\;\,$ &  &  &  \\ \hline
    \cellcolor[HTML]{FFFFFF} &  &  &  \\ \hline
    \end{tabular}
    $\cdots$
    \begin{tabular}{|l|l|l|
    >{\columncolor[HTML]{FFFFFF}}l |l|l|l|}
    \hline
     &  &  &  & \cellcolor[HTML]{C0C0C0}$\;\,$ & \cellcolor[HTML]{C0C0C0}$\cdots$ & \cellcolor[HTML]{C0C0C0}$\;\,$ \\ \hline
     &  &  & \cellcolor[HTML]{C0C0C0}$\;\,$ &  &  &  \\ \hline
    \cellcolor[HTML]{C0C0C0}$\;\,$ & \cellcolor[HTML]{C0C0C0}$\cdots$ & \cellcolor[HTML]{C0C0C0}$\;\,$ &  &  &  &  \\ \hline
    \end{tabular}
    $\cdots$
    \begin{tabular}{|l|l|l|l|}
    \hline
     &  &  & \cellcolor[HTML]{FFFFFF} \\ \hline
     &  &  & \cellcolor[HTML]{C0C0C0}$\;\,$ \\ \hline
    \rowcolor[HTML]{C0C0C0} 
    $\;\,$ & $\cdots$ & $\;\,$ & \cellcolor[HTML]{FFFFFF} \\ \hline
    \end{tabular}
\end{table}
By listing every possible combination of choices for the patterns, we are able to list all the possible histories; by converting the cells back into inequalities, this gives us the desired list of regions with a fixed label. Next, instead of finding a single region, we would like to determine the following.

\begin{question}
Given a Pak-Stanley label ${\bf p}$ for $P_n$, how can we find the minimal convex union of regions containing all the regions with label ${\bf p}$?
\end{question}

The method for solving this problem is similar to the first problem. First, we consider columns in which there is exactly one entry we can can choose corresponding to satisfying exactly one inequality of the form $x_i - x_{i+1} < 0$ (top), $0 < x_i - x_{i+1} < 1$ (middle), or $x_i - x_{i+1} > 1$ (bottom). Then, there are some columns in which we can either choose the bottom or middle entry and yield inequalities of the form $x_i - x_{i+1} > 0$. Similarly, there are some columns in which we can either choose the middle or top entry and yield inequalities of the form $x_i - x_{i+1} < 1$. Lastly, in columns in which we can choose any entry we obtain no inequalities. 

\begin{example}
Continuing Example \ref{ex:pak-to-hitory}, take the Pak-Stanley label $(0,1,1,0,2,0)$, which corresponds to the outcome $\{1,2,4,4\}$. By ``overlapping" the possible histories in Figure \ref{fig:all-possibilities}, we obtain the following board of possibilities for each column: 
\begin{figure}[H]
    \centering
    \begin{tabular}{lllll}
    \hline
    \multicolumn{1}{|l|}{0} & \multicolumn{1}{l|}{\cellcolor[HTML]{C0C0C0}1} & \multicolumn{1}{l|}{\cellcolor[HTML]{C0C0C0}2} & \multicolumn{1}{l|}{3} & \multicolumn{1}{l|}{\cellcolor[HTML]{C0C0C0}4} \\ \hline
    \multicolumn{1}{|l|}{\cellcolor[HTML]{C0C0C0}} & \multicolumn{1}{l|}{\cellcolor[HTML]{C0C0C0}} & \multicolumn{1}{l|}{\cellcolor[HTML]{C0C0C0}} & \multicolumn{1}{l|}{} & \multicolumn{1}{l|}{} \\ \hline
    \multicolumn{1}{|l|}{\cellcolor[HTML]{C0C0C0}1} & \multicolumn{1}{l|}{\cellcolor[HTML]{C0C0C0}2} & \multicolumn{1}{l|}{3} & \multicolumn{1}{l|}{\cellcolor[HTML]{C0C0C0}4} & \multicolumn{1}{l|}{5} \\ \hline
    \end{tabular}
    \caption{Minimal convex union of regions.}
    \label{fig:minimal-convex-union-of-regions}
\end{figure}
This yields the following list of inequalities, which gives the desired minimal convex union of regions: $x_0 - x_1 > 0$, $x_2 - x_3 < 1$, $x_3 - x_4 > 1$, and $x_4 - x_5 < 0$. 
\end{example}

\section{Trees and the $T$-Three Rows Game}\label{sec:five}
 
\subsection{Defining the $T$-Three Rows Game} 
In this section, we introduce the $T$-Three Rows Game to generalize our results from path graphs to tree graphs.
\begin{definition}[$T$-Three Rows Game] Let $T = (V,E)$ be a tree with $n$ vertices. The \textit{$T$-Three Rows Game} is played on a board with $n-1$ columns and 3 rows. Each column corresponds to an edge $\{i,j\}$: the top square has an $i$, the middle square is blank, and the bottom square has a $j$.

\begin{figure}[H]
\begin{tikzpicture}
\draw 
    (0,0) to (1.5,0) to (3,1) to (4.5,1)
    (3,1) to (1.5,2)
    ;
\fill 
    (0,0) circle (3pt) 
    (1.5,0) circle  (3pt) 
    (1.5,2) circle  (3pt) 
    (3,1) circle (3pt) 
    (4.5,1) circle (3pt); 
    \node at (0,-0.5) {$0$};
    \node at (1.5,-0.5) {$1$};
    \node at (1.5,1.5) {$3$};
    \node at (3,0.5) {$2$};
    \node at (4.5,0.5) {$4$};
    
    \node at (5.5,1) {$\Rightarrow$};
    
    \node at (7.5,1) {\begin{tabular}{|l|l|l|l|}
    \hline
    0 & 1 & 2 & 2 \\ \hline
     &  &  &  \\ \hline
    1 & 2 & 3 & 4 \\ \hline
    \end{tabular}};
\end{tikzpicture}
\end{figure}
As in the original Three Rows Game, exactly one square is chosen from each column of the board. Histories and outcomes are defined analogously.
\end{definition}


The same analogy as for path graphs holds between vertices (resp. labels) in the Shi adjacency digraph $\Gamma \mathscr{S}(T)$ and histories (resp. outcomes) of the $T$-Three Rows Game. Again, each column of the board corresponds to a pair of parallel hyperplanes, and our choice determines the region's relation to these hyperplanes. Because $T$ is acyclic, fixing any of these inequalities does not fix any of the other inequalities, just as in the case of the path graph. This yields a few basic properties:
\begin{proposition} Let $T$ be a tree with $n$ vertices. There  are $3^{n-1}$ vertices and $2^{n-1}$ sinks (which are all labelled with maximal $T$-parking functions) in the Shi adjacency digraph of $T$.
\end{proposition}

\begin{proposition}
Let $T$ be a tree. The multiplicities of the Pak-Stanley labels in the Shi adjacency digraph of $T$ are independent of the vertex labelling up to permutation. 
\end{proposition}

\subsection{Uniqueness of Maximal Outcomes} In this section, we generalize one of our main results of the previous sections, uniqueness of maximal outcomes for path graphs, to all trees.
\begin{theorem}[Uniqueness of Maximal Outcomes] Any outcome ${\bf o}$ of the $T$-Three Rows Game on $n$ columns with $|{\bf o}|=n$ has multiplicity $\mu({\bf o})=1$.
\end{theorem}
\begin{proof}
Notice that we can label the vertices of $T$ (using a depth-first search, for example) in such a way that each column in the $T$-Three Rows Game after the first contains exactly one vertex that has already appeared on the board, and exactly one vertex that has not appeared on the board.

Given an outcome ${\bf o}$, let $m_{{\bf o}}(i)$ denote the number of occurrences of $i$ in ${\bf o}$. Then, define \textit{distance} between outcomes as follows:
\begin{equation*}
    d({\bf o},{\bf o}') = \sum_{i = 0}^{n-1} |m_{\bf o}(i) - m_{{\bf o}'}(i)|.
\end{equation*}
Now, suppose that ${\bf h}$ and ${\bf h'}$ are two distinct histories of the $T$-Three Rows Game such that $|{\bf o}({\bf h})| = |{\bf o}({\bf h}')| = n$. Then, the history ${\bf h'}$ can be reached from ${\bf h}$ through a nonempty series of ``swaps" (that is, choosing the bottom square instead of the top, or the top square instead of the bottom) at some indices $i_1<\dots < i_k$. 

Yet the first swap $i_1$ increases the distance between ${\bf o}({\bf h})$ and ${\bf o}({\bf h'})$ by 2, and every subsequent swap $i_j$ does not decrease the distance between ${\bf o}({\bf h})$ and ${\bf o}({\bf h'})$. To see why, notice that each column of the board contains at least one new vertex $v'$ and one possibly old vertex $v$. Since this is the first time $v'$ appears, the swap ``displaces" $v'$, and even if it ``corrects" $v$, the distance is preserved. That is, after each swap, the distance $d({\bf o}({\bf h}),{\bf o}({\bf h}'))$ is increased either by $0$ or $2$. Thus, if ${\bf h} \neq {\bf h'}$, then $d({\bf o}({\bf h}),{\bf o}({\bf h'}))\geq 2$, and in particular ${\bf o}({\bf h}) \neq {\bf o}({\bf h'})$.
\end{proof}

\subsection{An Example: Star Graphs} In this section, we use the $T$-Three Rows Game to characterize the multiplicity of Pak-Stanley labels in the Shi arrangement for star graphs.

The star graph $S_n$ is a tree on $n+1$ vertices with a central vertex $0$ and $n$ ``spokes" $1,\dots,n$, all connected to the central vertex. Thus, the $S_n$-Three Rows Game is played on the following board: 
\begin{table}[H]
\begin{tabular}{|c|c|c|c|c|c|}
\hline
0 & 0 & 0 & $\cdots$ & 0 & 0 \\ \hline
 &  &  &  &  &  \\ \hline
1 & 2 & 3 & $\cdots$ & $n-1$ & $n$ \\ \hline
\end{tabular}
\end{table}
\begin{proposition}
Let ${\bf p}$ be a Pak-Stanley label in the $S_n$-Shi arrangement, and let $c_0$ (resp. $c_{\square}$) count the number of $0$s (resp. number of  middle $\square$s) in any history inducing the label. Then, 
\begin{equation*}
    \mu({\bf p}) = \binom{c_0+c_{\square}}{c_0}.
\end{equation*}
\end{proposition}
\begin{proof}
For each nonzero element $i$ in the outcome, there is exactly one column in the $S_n$-Three Rows Games where that element appears. The row choice is fixed for those columns. Since 0 and $\square$ appear in all columns, determining a history amounts to choosing $c_0$ columns out of the remaining $c_0 + c_\square$ unfixed columns to be 0 (leaving the remaining columns to be $\square$). Hence $\mu({\bf p}) = \binom{c_0+c_{\square}}{c_0}$.
\end{proof}
\begin{proposition}
The number of distinct Pak-Stanley labels of $S_n$ is \[\sum_{S\subseteq\{1,\hdots,n\}}n+1-|S|=(n+2)\cdot2^{n-1}.\]
\end{proposition}
\begin{proof}
Let $S_{\bf o}$ denote the set of non-zero elements of a given outcome ${\bf o}$. As in the previous proof, for any nonzero value appearing in the outcome, the corresponding column is fixed in the $S_n$-Three Rows Game. The remaining $n - |S_{\bf o}|$ columns are unfixed and can either contain $0$ or $\square$; this implies that for any ``nonzero set" $S \subseteq \{1,\dots,n\}$, there are $n+1-|S|$ distinct outcomes with said nonzero set. Therefore, the total number of distinct outcomes (i.e. distinct Pak-Stanley labels) is
\begin{align*}
    \sum_{S\subseteq\{1,\hdots,n\}}n+1-|S| &=
    (n+1)\cdot2^{n} - n\cdot2^{n-1}
    = (n+2)\cdot 2^{n-1}.
\end{align*}
\end{proof}

\section{Extending the Three Rows Game}\label{sec:six}
\subsection{Superstable Configurations on Cycle Graphs}
First, we compute the number of superstable configurations on a cycle graph. We begin with two easier determinant computations: 

\begin{proposition}
Let $M_n$ be the $n \times n$ matrix
\begin{equation*}
    \begin{bmatrix}
        3 & -1 & 0 & \cdots & 0 & 0 & 0 \\
        -1 & 3 & -1 & \cdots & 0 & 0 & 0 \\
        0 & -1 & 3 & \cdots & 0 & 0 & 0 \\
        \vdots & \vdots & \vdots & \ddots & \vdots & \vdots & \vdots \\
        0 & 0 & 0 & \cdots & 3 & -1 & 0 \\
        0 & 0 & 0 & \cdots & -1 & 3 & -1 \\
        0 & 0 & 0 & \cdots & 0 & -1 & 3
    \end{bmatrix}.
\end{equation*}
Then $\det(M_n) = F_{2n+2}$ for any positive integer $n$.
\end{proposition}
\begin{proof}
First, notice that $\det(M_1) = 3 = F_4$ and $\det(M_2) = 8 = F_6$. Therefore, it suffices to establish that $\det(M_n) = 3\det(M_{n-1}) - \det(M_{n-2})$ by Lemma \ref{lem:identities-of-fibonacci-numbers}. By cofactor expansion along the top row twice,
\begin{align*}
    \det(M_n)&=3\det(M_{n-1}) + \begin{vmatrix}
        -1 & -1 & \cdots & 0 & 0 \\
        0 & 3 & \cdots & 0 & 0 \\
        \vdots & \vdots & \ddots & \vdots & \vdots \\
        0 & 0 & \cdots & 3 & -1 \\
        0 & 0 & \cdots & -1 & 3 \\
    \end{vmatrix} = 3\det(M_{n-1}) - \det(M_{n-2}) + \begin{vmatrix}
        0 & \cdots & 0 & 0 \\
        \vdots & \ddots & \vdots & \vdots \\
        0 & \cdots & 3 & -1 \\
        0 & \cdots & -1 & 3 \\
        \end{vmatrix}.
\end{align*}
It follows that $\det(M_n) = 3\det(M_{n-1}) - \det(M_{n-2})$ as desired.
\end{proof}

\begin{proposition}
Let $M'_n$ be the $n \times n$ matrix   
\begin{equation*}
    \begin{bmatrix}
        0 & -1 & 0 & \cdots & 0 & 0 & 0 \\
        0 & 3 & -1 & \cdots & 0 & 0 & 0 \\
        0 & -1 & 3 & \cdots & 0 & 0 & 0 \\
        \vdots & \vdots & \vdots & \ddots & \vdots & \vdots & \vdots \\
        0 & 0 & 0 & \cdots & 3 & -1 & 0 \\
        0 & 0 & 0 & \cdots & -1 & 3 & -1 \\
        -1 & 0 & 0 & \cdots & 0 & -1 & 3
    \end{bmatrix}.
\end{equation*}
That is, $M'_n$ has a -1 in the $nth$ row, first column, 3's along the diagonal in rows $2,\cdots, n$. It has $-1$'s along the superdiagonal, $-1$'s along the subdiagonal in rows $3,\cdots, n$, and 0's elsewhere. Then $\det(M'_n) = -1$ for any positive integer $n$.
\end{proposition}
\begin{proof}
Notice $\det(M'_1) =-1$. Furthermore, by cofactor expansion along the top row,
\begin{equation*}
    \det(M'_n) =  \begin{vmatrix}
        0 & -1 & \cdots & 0 & 0 \\
        0 & 3 & \cdots & 0 & 0\\
        \vdots & \vdots & \ddots & \vdots & \vdots \\
        0 & 0 & \cdots & 3 & -1 \\
        -1 & 0 & \cdots & -1 & 3 \\
    \end{vmatrix}= \det(M'_{n-1}).
\end{equation*}
\end{proof}

Next, we have the following final determinant calculation: 

\begin{proposition}
Let $M''_n$ be the $n \times n$ matrix 
\begin{equation*}
    \begin{bmatrix}
        3 & -1 & 0 & \cdots & 0 & 0 & -1 \\
        -1 & 3 & -1 & \cdots & 0 & 0 & 0 \\
        0 & -1 & 3 & \cdots & 0 & 0 & 0 \\
        \vdots & \vdots & \vdots & \ddots & \vdots & \vdots & \vdots \\
        0 & 0 & 0 & \cdots & 3 & -1 & 0 \\
        0 & 0 & 0 & \cdots & -1 & 3 & -1 \\
        -1 & 0 & 0 & \cdots & 0 & -1 & 3
    \end{bmatrix}.
\end{equation*}
Then $\det(M''_n) = 3F_{2n} - 2F_{2n-2} - 2$.
\end{proposition}
\begin{proof}
First, by cofactor expansion along the top row, 
\begin{equation*}
    \det(M''_n)=3\det(M_{n-1}) + \begin{vmatrix}
        -1 & -1 & 0 & \cdots & 0 & 0 \\
        0 & 3 & -1 & \cdots & 0 & 0 \\
        0 & -1 & 3 & \cdots & 0 & 0 \\
        \vdots & \vdots & \vdots & \ddots & \vdots & \vdots \\
        0 & 0 & 0 & \cdots & 3 & -1 \\
        -1 & 0 & 0 & \cdots & -1 & 3 \\
    \end{vmatrix} + (-1)^n\begin{vmatrix}
        -1 & 3 & -1 & \cdots & 0 & 0 \\
        0 & -1 & 3 & \cdots & 0 & 0 \\
        0 & 0 & -1 & \cdots & 0 & 0 \\
        \vdots & \vdots & \vdots & \ddots & \vdots & \vdots \\
        0 & 0 & 0 & \cdots & -1 & 3 \\
        -1 & 0 & 0 & \cdots & 0 & -1 \\
    \end{vmatrix}.
\end{equation*}
For the second $(n - 1) \times (n - 1)$ determinant, cofactor expansion along the top row yields
\begin{equation*}
    -\det(M_{n-2}) + \det(M'_{n-2}) = -F_{2n-2} - 1.
\end{equation*}
For the third $(n - 1) \times (n - 1)$ determinant, cofactor expansion along the first column yields
\begin{align*}
    -\begin{vmatrix}
        -1 & 3 & \cdots & 0 & 0 & 0 \\
        0 & -1 & \cdots & 0 & 0 & 0 \\
        \vdots & \vdots & \ddots & \vdots & \vdots & \vdots \\
        0 & 0 & \cdots & -1 & 3 & -1 \\
        0 & 0 & \cdots & 0 & -1 & 3 \\
        0 & 0 & \cdots & 0 & 0 & -1
    \end{vmatrix} + (-1)^{n-1}\det(M_{n-2}).
\end{align*}
Now, the remaining $(n-2) \times (n-2)$ determinant is an upper triangular matrix, so its determinant can be computed as $(-1)^{n-2}$. Combining the previous together, we get that the third $(n-1) \times (n-1)$ determinant is $(-1)^{n-1} + (-1)^{n-1}\det(M_{n-2}) = (-1)^{n-1}F_{2n-2} + (-1)^{n-1}$. Therefore, 
\begin{equation*}
    \det(M''_n) = 3\det(M_{n-1})  - F_{2n-2} - 1 + (-1)^n((-1)^{n-1}F_{2n-2} + (-1)^{n-1}) = 3F_{2n} - 2F_{2n-2} - 2.
\end{equation*}
\end{proof}

\begin{theorem}
The number of superstable configurations of $(C_n)_\bullet$ is $3F_{2n} - 2F_{2n-2} - 2 = F_{2n+2} - F_{2n-2} - 2$.
\end{theorem}
\begin{proof}
In the same vein as Theorem \ref{thm:superstable-configurations-of-P_n}, this follows from the above determinant calculation, and Theorem 4.2.2 of Klivans' book in light of the definition of the reduced Laplacian \cite{Klivans}. The equality follows from the second identity in Lemma \ref{lem:identities-of-fibonacci-numbers}.
\end{proof}

\begin{lemma}[Superstability Criterion for Cycle Graphs]\label{lem:Carls-criterion-for-cycle-graphs}
Let ${\bf a}\in \{0,1,2\}^n$. 
Consider the entries of ${\bf a}$ arranged on a circle so that entries $a_1$ and $a_n$ are adjacent. Suppose ${\bf a}$ has 2-free blocks $\pi_1,\pi_2,\cdots,\pi_m$ so that $\pi_1=\pi_m$. Then ${\bf a}$ is a superstable configuration on $(C_n)_{\bullet}$ if and only if there is a $0$ in each $2$-free block ${\bf \pi}_i$.
\end{lemma}
\begin{proof}
The proof follows just as in Lemma \ref{lem:Carls-criterion}.
\end{proof}

\subsection{The Cyclical Three Rows Game}
Next, we define the Cyclical Three Rows Game to analyze multiplicity in the Pak-Stanley labels for the cycle graph $C_n$. The Cyclical Three Rows is similar to the $T$-Three Rows Game with only a minor modification: the concept of \textit{legal} and \textit{illegal} histories.

\begin{definition}[Cyclical Three Rows Game]\label{def:cyclical-three-rows-game}
The \textit{Cyclical Three Rows Game} is played on the following board: 
\begin{center}
    \begin{tabular}{|l|l|l|l|l|l|l|}
    \hline
    0 & 1 & 2 & $\cdots$ & $n-3$ & $n-2$ & 0 \\ \hline
     &  &  &  &  &  &  \\ \hline
    1 & 2 & 3 & $\cdots$ & $n-2$ & $n-1$ & $n-1$ \\ \hline
    \end{tabular}
\end{center}
As in the original Three Rows Game, exactly one square is chosen from each column of the board, with outcomes and histories being defined analogously. Furthermore, every choice except the final one is called a \textit{path choice}, and the final choice is called the \textit{cycle choice}.
\end{definition}

To understand the idea of legal and illegal histories we will consider an example.

\begin{example}
Consider the following history of the Cyclical Three Rows Game on $3$ columns: 
\begin{table}[H]
    \centering
    \begin{tabular}{|l|l|l|}
    \hline
    \rowcolor[HTML]{C0C0C0} 
    0 & 1 & \cellcolor[HTML]{FFFFFF}0 \\ \hline
    \rowcolor[HTML]{FFFFFF} 
     &  &  \\ \hline
    \rowcolor[HTML]{FFFFFF} 
    1 & 2 & \cellcolor[HTML]{C0C0C0}2 \\ \hline
    \end{tabular}
    \label{tab:example-history-of-the-cyclical-three-rows-game}
\end{table}
If we convert this history into the sequence of corresponding inequalities, we get the following: 
\begin{equation*}
    x_0 - x_1 < 0 \qquad x_1 - x_2 < 0 \qquad x_0 - x_2 > 1.
\end{equation*}
However, the first two inequalities yield $x_0 - x_2 < 0$, a contradiction to the last inequality. Therefore, this history determines a region which does not exist. This is the core idea of illegal histories: because the analogy between histories and regions breaks down, we define a particular class of histories, ``legal histories", which \textit{do} preserve this analogy.
\end{example}

\begin{definition}[Legal History]\label{def:legal-history}
A \textit{legal history (of the Cyclical Three Rows Game)} is a history such that, when the history is converted to a list of inequalities defining a region using the following rules, the region is nonempty (that is, the system of inequalities can be solved): 
\begin{enumerate}
    \item A choice of the form \begin{tabular}{|l|}
    \hline
    \rowcolor[HTML]{C0C0C0} 
    $i$ \\ \hline
    \rowcolor[HTML]{FFFFFF} 
     \\ \hline
    \rowcolor[HTML]{FFFFFF} 
    $j$ \\ \hline
    \end{tabular} is converted to the inequality $x_i - x_j < 0$.
    \item A choice of the form \begin{tabular}{|l|}
    \hline
    \rowcolor[HTML]{FFFFFF} 
    $i$ \\ \hline
    \rowcolor[HTML]{C0C0C0} 
     \\ \hline
    \rowcolor[HTML]{FFFFFF} 
    $j$ \\ \hline
    \end{tabular} is converted to the inequality $0 < x_i - x_j < 1$.
    \item A choice of the form \begin{tabular}{|l|}
    \hline
    \rowcolor[HTML]{FFFFFF} 
    $i$ \\ \hline
    \rowcolor[HTML]{FFFFFF} 
     \\ \hline
    \rowcolor[HTML]{C0C0C0} 
    $j$ \\ \hline
    \end{tabular} is converted to the inequality $x_i - x_j > 1$.
\end{enumerate}
An \textit{illegal history} is a history of the Cyclical Three Rows Game which is not legal.
\end{definition}

Next, let us provide an alternate classification of legal histories.

\begin{proposition}[Classification of Legal Histories]\label{prop:classification-of-legal-histories}
Let ``$T$" denote the top choice, ``$M$" denote the middle choice, and ``$B$" denote the bottom choice, so that a history of the Cyclical Three Rows Game on $n$ columns is a tuple ${\bf h} = \{T,M,B\}^n$. Then, a history ${\bf h} \in \{T,M,B\}^n$ of the Cyclical Three Rows Game is a legal history if and only if it follows these three rules: 
\begin{enumerate}
    \item If all of the path choices are $T$, the cycle choice must be $T$.
    \item If all of the path choices are $T$ except for one $M$, the cycle choice cannot be $B$.
    \item If all of the path choices are $M$ or $B$, the cycle choice must be $B$ unless \textit{every} path choice is $M$ (in which case the cycle choice can also be $M$).
\end{enumerate}
\end{proposition}
\begin{proof}
There are a few insights which make our casework simpler.
Firstly, notice that the path choices are independent, so it suffices to determine when particular cycle choices are impossible. Secondly, if $T$ and $B$ both occur as path choices, then the remaining path choices cannot be used to create an inequality that restricts the cycle choice. This leaves us with the following cases: 
\begin{enumerate}
    \item All of the path choices are $T$.
    \item All of the path choices are $T$ or $M$.
    \item All of the path choices are $M$.
    \item All of the path choices are $M$ or $B$.
    \item All of the path choices are $B$.
\end{enumerate}
In the first case, notice that by summing the series of inequalities given by all of our path choices being $T$, namely $x_0 - x_1 < 0, x_1 - x_2 < 0, \dots, x_{n-2} - x_{n-1} < 0$, we achieve the inequality $x_0 - x_{n-1} < 0$. This forces the cycle choice to be $T$. In the second case, notice that if there are at least two $M$s, then the cycle choice can be $T$, $M$, or $B$. Therefore, the only interesting case is that of a single $M$; in this case, the cycle choice can be $T$ or $M$, but not $B$. In the third case, notice that the cycle choice can be $M$ or $B$ but not $T$. In the fourth and fifth cases, notice that if all the path choices are $M$ or $B$, and there is at least one $B$, then the cycle choice \textit{must} be $B$. Reorganizing gives us the list of three rules in the statement of the theorem, as desired.
\end{proof}

This classification allows us to deduce the following facts: 

\begin{corollary}
There are $2^n - 2$ sinks in the Shi adjacency graph of $C_n$.
\end{corollary}
\begin{proof}
The Classification of Legal Histories demonstrates that whenever there is an $M$ in a legal history, it can be transformed into a $T$ or a $B$ (and sometimes both) to result in another legal history. In other words, the sinks correspond to the legal histories without $M$s. Now, there are $2^n$ such histories, and the Classification of Legal Histories demonstrates that there are two illegal ones: $T \cdots TB$ and $B \cdots BT$. This leaves $2^n - 2$ such legal histories.
\end{proof}

\begin{corollary}
There are $3^n - 2^n - n$ legal histories of the Cyclical Three Rows Game on $n$ columns; in particular, there are $3^n - 2^n - n$ regions in the $C_n$-Shi arrangement.
\end{corollary}
\begin{proof}
This follows from simple counting. There are $3^n$ histories of the Cyclical Three Rows Game. The first rule yields $2$ illegal histories, the second rule yields $n-1$ illegal histories, and the third rule yields $2^{n-1} \cdot 2 - 1$ illegal histories for a total of $2 + (n - 1) + (2^n - 1) = 2^n + n$ illegal histories.
\end{proof}

Finally, our classification of the illegal histories in the Cyclical Three Rows Game, in conjunction with Theorem \ref{thm:pattern-repetition-theorem}, can be used to give a method for computing the set of regions $\mathscr{R}$ which give rise to a given Pak-Stanley label for $C_n$, as follows: 
\begin{enumerate}
    \item Given a Pak-Stanley label ${\bf p}$ for $C_n$, convert it to an outcome ${\bf o}$ of the Cyclical Three Rows Game on $n$ columns.
    \item Using the Pattern Multiplicity Theorem, enumerate the histories of the Three Rows Game on $n-1$ columns inducing ${\bf o}$. Check which of these path choices are legal when a cycle choice $M$ is appended. Convert the ones which pass back into regions and add them to $\mathscr{R}$.
    \item If ${\bf o}$ has a $0$, create a new outcome ${\bf o}'$ which is identical to ${\bf o}$ with one less $0$. Then, using the Pattern Multiplicity Theorem, enumerate the histories of the Three Rows Game on $n-1$ columns inducing ${\bf o}'$. Check which of these path choices are legal when a cycle choice $T$ is appended. Convert the ones which pass back into regions and add them to $\mathscr{R}$.
    \item If ${\bf o}$ has an $n-1$, create a new outcome ${\bf o}''$ which is identical to ${\bf o}$ with one less $n-1$. Then, using the Pattern Multiplicity Theorem, enumerate the histories of the Three Rows Game on $n-1$ columns inducing ${\bf o}''$. Check which of these path choices are legal when a cycle choice $B$ is appended. Convert the ones which pass back into regions and add them to $\mathscr{R}$.
\end{enumerate}
The final list $\mathscr{R}$ will contain precisely those regions of $C_n$ with the Pak-Stanley label ${\bf p}$. This procedure is not quite as simple as it is for the Three Rows Game, but it can be done with only a slight slowdown, and is certainly much better than a naive search.

\subsection{Defining the $G$-Three Rows Game}
The $G$-Three Rows Game combines the lessons from studying the $T$-Three Rows Game and the Cyclical Three Rows Game into a single concept which subsumes both of them. 

\begin{definition}[$G$-Three Rows Game]
Suppose that $G$ is a graph with $n$ vertices and $m$ edges. Then, the $G$-Three Rows Game is played on $m$ columns and $3$ rows. Each column corresponds to an edge $\{i,j\}$; the top square has an $i$, the middle square is blank, and the bottom square has a $j$. Subsequent definitions, including \textit{(legal and illegal) histories}, \textit{outcomes}, \textit{multiplicity}, and so on are defined analogously.
\end{definition}


Immediately, there are a few open problems which we have yet to answer (either with an efficient algorithm or a formula): 

\begin{problem}
Given a graph $G$, how many legal histories of the $G$-Three Rows Game that do not use the middle square are there?
\end{problem}

\begin{problem}
Given a graph $G$, how many legal histories of the $G$-Three Rows Game are there?
\end{problem}

\begin{problem}
Given a legal history ${\bf h}$ of the $G$-Three Rows Game, what is $\mu({\bf h})$?
\end{problem}

The last two questions correspond to the problems of (1) determining the number of regions in the $G$-Shi arrangement and (2) determining the number of times a given $G_\bullet$-parking functions appears as a Pak-Stanley label of said regions, respectively. However, despite the fact that these fundamental questions (among others) are still open, we are able to prove an interesting fact about the $G$-Three Rows Game: every maximal outcome appears uniquely. 

\subsection{Uniqueness of Maximal Outcomes}
In this section, we give a proof of uniqueness for maximal outcomes in the $G$-Three Rows Game, analogous to the proofs of uniqueness of maximal outcomes in the $T$-Three Rows Game and ordinary Three Rows Game given earlier. This theorem generalizes both of those facts.

\begin{proposition}
Suppose that $G$ is a graph with $m$ edges and ${\bf h}$ is a history of the $G$-Three Rows Game. Then ${\bf o}$ is a maximal outcome (i.e. corresponds to a maximal $G_\bullet$-parking function ${\bf p}$) if and only if ${\bf h}$ does not ever choose the middle square.
\end{proposition}
\begin{proof}
By Proposition \ref{prop:maximal-superstable-conf} and Theorem \ref{thm:superstable-configurations-are-G-parking-functions}, ${\bf p}$ is maximal if and only if $|{\bf o}| = m$ if and only if we choose a square with a number (as opposed to a middle square, which has no number) in each column of the $G$-Three Rows Game.
\end{proof}

\begin{theorem}
Let $G$ be a graph and ${\bf o}$ a maximal outcome of the $G$-Three Rows Game. Then $\mu({\bf o}) = 1$.
\end{theorem}
\begin{proof}
Suppose ${\bf h}$ is a legal history of the $G$-Three Rows Game with maximal outcome ${\bf o}$, and ${\bf h}'$ is another history with ${\bf h}' \sim {\bf h}$. Since ${\bf o}$ is maximal, both ${\bf h}$ and ${\bf h}'$ only use the top and bottom rows, so ${\bf h}$ can be transformed into ${\bf h}'$ by a sequence of ``swaps" from a top choice to a bottom choice and vice versa. Since ${\bf h} \neq {\bf h}'$, there is at least one swap, say $a_1 \to a_2$. Then the region $R$ corresponding to ${\bf h}'$ satisfies $x_{a_1} > x_{a_2}$. Yet now $a_2$ appears too many times in ${\bf o}({\bf h}')$, so there must be another swap $a_2 \to a_3$ (so $R$ satisfies $x_{a_2} > x_{a_3}$). Now $a_3$ appears too many times in ${\bf o}({\bf h}')$, so there must be another swap $a_3 \to a_4$, and so on. Since there are only finitely many possible swaps, $a_i = a_j$ for some $i < j$. But then $R$ satisfies the unsolvable sequence of inequalities $x_{a_i} > x_{a_{i+1}} > \cdots > x_{a_j} = x_{a_i}$, so $R$ is empty and ${\bf h}'$ is illegal, as desired.
\end{proof}

\begin{corollary}
Let $G$ be a graph and ${\bf p}$ a maximal $G_\bullet$-parking function. Then $\mu({\bf p}) = 1$.
\end{corollary}

However, this raises a question which is particular to our new, generalized, situation. For it is still clear that if $v$ is labelled with a maximal $G_\bullet$-parking function, then $v$ is a sink of the Shi adjacency digraph, but the converse is no longer clear. The existence and prevalence of illegal histories means that it is not clear that we can always ``eliminate any occurrence of blank middle squares"; as a result, it is not clear that every sink of the Shi adjacency digraph of $G$ is labelled with a maximal $G_\bullet$-parking function. This is our final open problem: 

\begin{problem}
If $G$ is an arbitrary graph, is every sink of the Shi adjacency digraph of $G$ labelled with a maximal $G_\bullet$-parking function?
\end{problem}

\section{Conclusion}
There is still much to discover about the nature of multiplicities in the Pak-Stanley labels of $G$-Shi arrangements. In particular, a better understanding of the ways that cycles interact could lead to solutions to the open problems discussed in Section 6.3. Even in the acyclic case, there is much room to study different classes of trees or find a fast algorithms for playing the $T$-Three Rows Game for all trees. Nonetheless, our work provides the first step towards understanding the complicated relationship between $G_\bullet$-parking functions, superstable configurations, and regions in the $G$-Shi arrangement provides by the Pak-Stanley algorithm, and we hope that it encourages further interest in the topic.

\section*{Acknowledgements}
This material is based upon work supported by the National Science Foundation under Grant No. DMS-1929284 while the authors were in residence at the Institute for Computational and Experimental Research in Mathematics (ICERM) in Providence, RI, during the Summer@ICERM program.
The authors would like to thank the Summer@ICERM program for bringing us together to work on this project and ICERM for hosting us as we completed our research. We also thank Pamela E. Harris and Susanna Fishel for their mentorship and guidance.

\printbibliography
\appendix
\section{Computational Tools for Analyzing Shi Adjacency Graphs}\label{appendix:g_shi_code}

\subsection{Method Descriptions}

The \lstinline{G_Shi} class takes in a Graph object and uses it to generate the Shi adjacency digraph and Pak-Stanley labels. In this section, we will describe the functionality of some of the primary methods.

\begin{itemize}
    \item \lstinline{make_adjacency_graph()}

    Returns the undirected Shi adjacency graph. 
    
    This function creates a list of the visited and unvisited regions in the $G$-Shi arrangement. It then iterates through the unvisited regions list, and checks which visited regions the unvisited region is adjacent to. Each pair of adjacent regions is stored in a list. A graph is created with the list of all regions as vertices, and the list of adjacent regions as edges. 

    Example:
\begin{lstlisting}
G = graphs.CompleteGraph(3)
Q = G_Shi(G)
adjacency_graph = Q.make_adjacency_graph()
adjacency_graph.show(spring=True, vertex_labels=False)
\end{lstlisting}
    Output:
    \begin{figure}[H]
        \centering
        \includegraphics[width=0.3\textwidth]{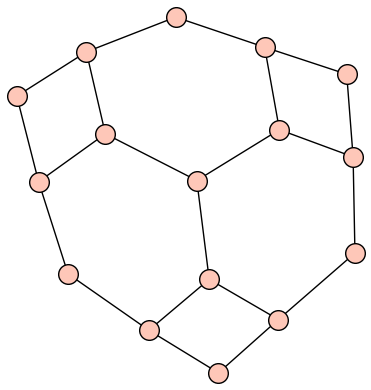}
    \end{figure}
    
    \item \lstinline{make_adjacency_digraph()}

    Returns the Shi adjacency digraph.

    This function iterates through the edges of the undirected Shi adjacency graph and creates a list of directed edges. For each undirected edge, it directs the new edge depending on the relative position of the base region. It creates and returns a digraph using the original vertices and a new list of directed edges. 
    
    Example:
\begin{lstlisting}
G = graphs.CompleteGraph(3)
Q = G_Shi(G)
adjacency_digraph = Q.make_adjacency_digraph()
adjacency_digraph.show(spring=True, vertex_labels=False)
\end{lstlisting}
    Output:
    \begin{figure}[H]
        \centering
        \includegraphics[width=0.3\textwidth]{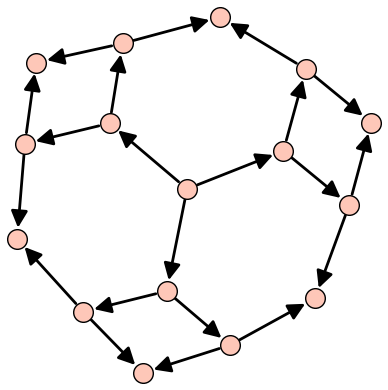}
    \end{figure}
    
    \item \lstinline{assign_edge_labels()}
    
    Returns a dictionary with the edges and edge labels.
    
    This helper method iterates through the edges of the Shi adjacency digraph and assigns a label to each edge. If the edge corresponds to the hyperplane $x_i-x_j=0$, the edge is labelled $i$. If the edge corresponds to the hyperplane $x_i-x_j=1$, the edge is labelled $j$. 
    
    \item \lstinline{make_pak_stanley_labels()}
    
    Returns a dictionary with the regions and Pak-Stanley labels.
    
    This helper method iterates through the vertices of the Shi adjacency digraph and assigns a label to each vertex using the Pak-Stanley algorithm.

    \item \lstinline{make_pak_stanley_graph_of_regions()}
    
    Returns the Shi adjacency digraph with Pak-Stanley labels on the vertices. 
    
    This function creates the Shi adjacency digraph and assigns a label to each edge in the graph. To avoid issues with repeated labels, it uses the inner class \lstinline{AdvancedVertex} to replace the existing vertices. 

    Example:
\begin{lstlisting}
G = graphs.CompleteGraph(3)
Q = G_Shi(G)
pak_stanley_graph = Q.make_pak_stanley_graph_of_regions()
pak_stanley_graph.show(spring=True)
\end{lstlisting}
    Output:
    \begin{figure}[H]
        \centering
        \includegraphics[width=0.3\textwidth]{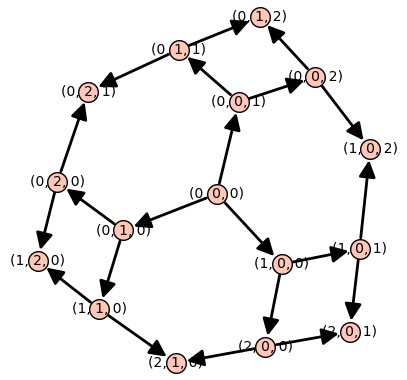}
    \end{figure}
    
    \item \lstinline{count_labels()}
    
    Returns a dictionary with the number of occurrences of each Pak-Stanley label.
    
    This function iterates through the list of Pak-Stanley labels and adds one to the count for each occurrence of a given label. 
    
    \textbf{Example:}
\begin{lstlisting}
G = graphs.CompleteGraph(3)
Q = G_Shi(G)
print(Q.count_labels())
\end{lstlisting}
    \textbf{Output:}
    
    \lstinline!{(0, 0, 0): 1, (0, 0, 1): 1, (0, 1, 0): 1, (1, 0, 0): 1, (1, 1, 0): 1, (0, 2, 0): 1, (2, 1, 0): 1, (1, 2, 0): 1, (0, 0, 2): 1, (0, 1, 1): 1, (0, 2, 1): 1, (0, 1, 2): 1, (1, 0, 2): 1, (2, 0, 0): 1, (1, 0, 1): 1, (2, 0, 1): 1}!

\end{itemize}

\subsection{Code}
Displayed below is our code for the \lstinline{G_Shi} class. 

\begin{lstlisting}
class G_Shi:

    def init(self, graph):
        """Initializes an instance of the class for a given graph."""
        self.arrangement = hyperplane_arrangements.G_Shi(graph)
        self.regions = self.arrangement.regions()
        self.n = len(self.regions[0].vertices()[0])
        self.base_region = self.find_base_region()
        self.graph_of_regions = self.make_adjacency_graph()
        self.digraph_of_regions = self.make_adjacency_digraph()
        self.pak_stanley_edge_labels = self.assign_edge_labels()
        self.pak_stanley_labels = self.make_pak_stanley_labels()
        self.pak_stanley_graph_of_regions = self.make_pak_stanley_graph_of_regions()
        self.label_count = self.count_labels()

    def is_base_region(self, region):
        """Checks if this region is the base region."""
        return region.contains(((self.n - i) / self.n for i in range(self.n)))

    def find_base_region(self):
        """Uses is_base_region to find and return the base region."""
        for region in self.regions:
            if self.is_base_region(region):
                return region
        return False

    def share_face(self, region1, region2):
        """Checks if two regions share a face."""
        if self.arrangement.distance_between_regions(region1, region2) == 1:
            return True
        return False

    def make_adjacency_graph(self):
        """Creates the undirected Shi adjacency graph."""
        unvisited_regions = list(self.regions)
        unvisited_regions.remove(self.base_region)
        vertices = [self.base_region]
        edges = []
        # Iterate through all of the regions in the arrangement
        while unvisited_regions:
            for region in unvisited_regions:
                for vertex in vertices:
                    # Check if two regions are adjacent
                    if self.share_face(region, vertex):
                        if region not in vertices:
                            unvisited_regions.remove(region)
                            vertices.append(region)
                        edges.append((vertex, region))
        return Graph([vertices, edges])

    def boundary_inequality(self, region1, region2):
        """Returns the inequality which divides two adjacent regions, or False if no such inequality exists."""
        for ineq1 in region1.inequalities():
            for ineq2 in region2.inequalities():
                if ineq1.A() == -1 * ineq2.A() and ineq1.b() == -1 * ineq2.b():
                    return ineq1
        return False

    def make_adjacency_digraph(self):
        """Creates the Shi adjacency digraph."""
        # Initialize the graph
        initial_graph = self.graph_of_regions
        vertices = initial_graph.vertices()
        undirected_edges = initial_graph.edges()
        directed_edges = []
        # Define the base point to be used to find the location of the base region
        base_point = vector((self.n - i) / self.n for i in range(self.n))
        # Direct the edges
        for edge in undirected_edges:
            inequality = self.boundary_inequality(edge[0], edge[1])
            if inequality.eval(base_point) >= 0:
                directed_edges.append((edge[0], edge[1]))
            else:
                directed_edges.append((edge[1], edge[0]))
        return DiGraph([vertices, directed_edges])

    def assign_edge_labels(self):
        """Finds the index that is incremented along each edge in the Shi adjacency digraph."""
        digraph = self.digraph_of_regions
        edges = digraph.edges()
        edge_labels = {}
        # Label each edge depending on the corresponding inequality
        for edge in edges:
            inequality = self.boundary_inequality(edge[0], edge[1])
            if inequality.b() == 0:
                for i in range(len(inequality.A())):
                    if inequality.A()[i] != 0:
                        edge_labels[edge] = i
                        break
            else:
                for i in range(len(inequality.A())):
                    if inequality.A()[len(inequality.A()) - 1 - i] != 0:
                        edge_labels[edge] = len(inequality.A()) - 1 - i
                        break
        return edge_labels

    def make_pak_stanley_labels(self):
        """Assigns Pak-Stanley labels to vertices of the Shi adjacency digraph."""
        digraph = self.digraph_of_regions
        unvisited_vertices = digraph.vertices()
        unvisited_vertices.remove(self.base_region)
        vertex_labels = {self.base_region: [0] * self.n}
        new_vertices = [self.base_region]
        # Assign the labels iterating outwards using a breadth-first traversal of the graph
        while unvisited_vertices:
            for vertex in new_vertices:
                new_vertices.remove(vertex)
                for edge in digraph.outgoing_edges(vertex):
                    if edge[1] in unvisited_vertices:
                        vertex_labels[edge[1]] = vertex_labels[edge[0]].copy()
                        vertex_labels[edge[1]][self.pak_stanley_edge_labels[edge]] += 1
                        new_vertices.append(edge[1])
                        unvisited_vertices.remove(edge[1])
        for vertex in vertex_labels:
            vertex_labels[vertex] = tuple(vertex_labels[vertex])
        return vertex_labels

    def make_pak_stanley_graph_of_regions(self):
        """Creates the Shi adjacency digraph with Pak-Stanley vertex labels."""

        class AdvancedVertex(tuple):
            def __init__(self, v):
                self.vertex = v

            def str(self):
                return str(self.vertex[-1])

        digraph = self.digraph_of_regions
        # Create basic graph with non-injective vertex labels
        vertices = [AdvancedVertex((vertex, self.pak_stanley_labels[vertex])) for vertex in digraph.vertices()]
        edges = [(AdvancedVertex((edge[0], self.pak_stanley_labels[edge[0]])), AdvancedVertex((edge[1], self.pak_stanley_labels[edge[1]]))) for edge in digraph.edges()]
        graph = DiGraph([vertices, edges])
        # Add edge labels
        for edge in digraph.edges():
            graph.set_edge_label(AdvancedVertex((edge[0], self.pak_stanley_labels[edge[0]])), AdvancedVertex((edge[1], self.pak_stanley_labels[edge[1]])), self.pak_stanley_edge_labels[edge])
        return graph

    def count_labels(self):
        """Counts the number of occurrences of each Pak-Stanley label."""
        labels = [tuple(label) for label in self.pak_stanley_labels]
        counter = {}
        for label in labels:
            if label in counter:
                counter[label] += 1
            else:
                counter[label] = 1
        return counter
\end{lstlisting}

\section{Computing Superstable Configurations for $P_n$}\label{subsec:app-b-the-path-graph}
In this section, we detail another proof of Theorem \ref{thm:superstable-configurations-of-P_n}. First, we establish a basic identity of Fibonacci numbers to be used later.

\begin{lemma}[Even Fibonacci Numbers]\label{lem:an-identity-for-even-fibonacci-numbers}
Let $F_n$ denote the $n$th Fibonacci number. Then, 
\begin{equation*}
    F_{2n} = 2F_{2n-2} + F_{2n-4} + F_{2n-6} + \cdots + F_4 + 2F_2.
\end{equation*}
for any positive integer $n$.
\end{lemma}
\begin{proof}
This result follows by repeated application of the identity $F_n = F_{n-1} + F_{n-2}$; indeed, 
\begin{align*}
    F_{2n} &= F_{2n-1} + F_{2n-2} = 2F_{2n-2} + F_{2n-3} = 2F_{2n-2} + F_{2n-4} + F_{2n-5} \\ &= 2F_{2n-2} + F_{2n-4} + \cdots + F_4 + F_3 = 2F_{2n-2} + F_{2n-4} + \cdots + F_4 + 2F_2
\end{align*}
where the penultimate equality comes from repeated application of the recurrence relation to the smallest Fibonacci number remaining, and the final equality comes from $F_3 = 2 = 2F_2$.
\end{proof}

The essential idea of this proof is that if a sequence $\{a_n\}$ satisfies the initial condition $a_1 = F_2$ and the recurrence $a_n = 2a_{n-1} + a_{n-2} + \cdots + a_2 + 2a_1$, then necessarily we have $a_n = F_{2n}$. This fact follows from Lemma \ref{lem:an-identity-for-even-fibonacci-numbers}. 

\begin{theorem}
The number of superstable configurations of $(P_n)_\bullet$ is $F_{2n}$ for any positive integer $n$.
\end{theorem}
\begin{proof}
Let $s_n$ denote the number of superstable configurations of $(P_n)_\bullet$. Then $s_1 = 1 = F_2$; the only superstable configuration on $(P_1)_\bullet$ is $(0)$. Therefore, it suffices to establish the recurrence $s_n = 2s_{n-1} + s_{n-2} + \cdots + s_2 + 2s_1.$\\ \\
For this, we will partition the collection of superstable configurations. Say that a superstable configuration ${\bf a} = (a_1,\dots,a_n)$ \textit{belongs} to bucket $i$ if $i$ is the largest non-negative integer strictly less than $n$ such that $(a_1,\dots,a_i)$ is a superstable configuration on $(P_i)_\bullet$. Then, we will show that (1) there are $2s_{n-1}$ configurations in bucket $n-1$, (2) there are $s_i$ configurations in bucket $i$ for $i \in \{1,\dots,n-2\}$, and (3) there are $s_1 = 1$ configurations in bucket $0$. This suffices to establish the recurrence relation, so we will prove these facts now:\\ \\
\textbf{(1): }Suppose that ${\bf a} = (a_1,\dots,a_n)$ is a superstable configuration belonging to bucket $n-1$. Then there are $s_{n-1}$ options for $(a_1,\dots,a_{n-1})$. By the Superstability Criterion, we cannot have $a_n = 2$ but choosing $a_n = 0$ or $a_n = 1$ results in a superstable configuration regardless of which superstable configuration is chosen for $(a_1,\dots,a_{n-1})$). Therefore, there are $2s_{n-1}$ configurations in this bucket.\\ \\
\textbf{(2): }Suppose that ${\bf a} = (a_1,\dots,a_n)$ is a superstable configuration belonging to bucket $i$ for some $i \in \{1,\dots,n-2\}$. In particular, $(a_1,\dots,a_i)$ is superstable but $(a_1,\dots,a_{i+1})$ is not. Now, by the Superstability Criterion, we must have $a_{i+1} = 2$. Then, there cannot exist any $j > i+1$ with $a_j = 2$, since then by the Superstability Criterion there would exist a $0$ between $a_{i+1}$ and $a_j$ and then $(a_1,\dots,a_{j-1})$ is a superstable configuration. Yet since $j-1 > i$, this contradicts the fact that ${\bf a}$ belongs to bucket $i$. Similarly, there cannot exist any $i+1 < j < n$ with $a_j = 0$, since then by the Superstability Criterion $(a_1,\dots,a_j)$ is a superstable configuration, a contradiction for the same reason. Thus, $a_{i+2},\dots,a_{n-1} = 1$, and the Superstability Criterion forces $a_n = 0$. In summary, superstable configurations in bucket $i$ are uniquely determined by the superstable configuration consisting of their first $i$ entries. Therefore, there are $s_i$ configurations in this bucket.\\ \\
\textbf{(3): }Suppose that ${\bf a} = (a_1,\dots,a_n)$ is a superstable configuration belonging to bucket $0$. Now, $a_1$ cannot be $2$ (since then ${\bf a}$ is not superstable), but in fact there cannot exist any $i > 1$ with $a_i = 2$. To see why, notice that by the Superstability Criterion there would exist a $0$ before $a_i$ and then $(a_1,\dots,a_{i-1})$ would be superstable, a contradiction with the fact that ${\bf a}$ belongs to bucket $0$. Thus ${\bf a}$ consists of only $0$s and $1s$. Yet then there cannot exist any $i < n$ with $a_i = 0$, as then by the Superstability Criterion $(a_1,\dots,a_i)$ is superstable, a contradiction for the same reason. On the other hand, the Superstability Criterion requires that there is at least one zero. This forces ${\bf a} = (1,\dots,1,0)$, so there is a unique such configuration as desired.
\end{proof}
\end{document}